\documentclass[a4paper]{amsart}

\usepackage{amsmath, amssymb, amsthm, comment}
\usepackage{a4wide,rotating,enumitem}

\numberwithin{equation}{section}

\newcommand\Cop{\Delta_s} 
\newcommand\COP{\Delta_u} 
\newcommand\K{\mathcal{K}}
\newcommand\coP{\Delta_{r}}

\usepackage{dsfont}
\newcommand{\WW}{{\mathds{V}\!\!\text{\reflectbox{$\mathds{V}$}}}}
\newcommand{\Ww}{\mathds{W}}

\newcommand{\ww}{\textup{W}}
\newcommand{\vV}{\reflectbox{$\mathds{V}$}}
\newcommand{\vv}{\textup{V}}

\theoremstyle{plain}
\newtheorem{theorem}{Theorem}[section]
\newtheorem{tw}[theorem]{Theorem}
\newtheorem{lem}[theorem]{Lemma}
\newtheorem{prop}[theorem]{Proposition}
\newtheorem{cor}[theorem]{Corollary}

\theoremstyle{remark}
\newtheorem{remark}[theorem]{Remark}

\newtheorem{deft}[theorem]{Definition}

\DeclareMathOperator\supp{supp}
\DeclareMathOperator\ball{ball}

\DeclareMathOperator\linsp{span}

\newcommand{\lt}{L}

\newcommand\sm{\setminus}
\newcommand\col{\colon}
\newcommand\sub{\subseteq}

\newcommand\Ind{\mathcal{I}}

\newcommand{\N}{\mathbb{N}}

\newcommand{\complex}{\mathbb{C}}
\newcommand{\bn}{\mathbb{N}}
\newcommand{\bc}{\mathbb{C}}
\newcommand{\bt}{\mathbb{T}}
\newcommand{\C}{\mathbb{C}}

\DeclareMathOperator{\Fix}{Fix}

\DeclareMathOperator\Prob{Prob}
\DeclareMathOperator\Ran{Ran}

\DeclareMathOperator\clsp{\hbox{$\overline{\mathrm{span}}$}}

\newcommand\conv{\star}
\newcommand\latetilde{\widetilde{\phantom{x}}}
\DeclareMathOperator\plim{\text{$p$}-lim}
\DeclareMathOperator*\wlim{w*-lim}

\newcommand{\B}{\mathrm{B}}

\newcommand\ruc{\mathit{RUC}}
\newcommand\luc{\mathit{LUC}}
\newcommand\uc{\mathit{UC}}

\newcommand{\WAP}{\mathit{WAP}}

\newcommand{\set}[2]{\{\,{\textstyle#1}:\,{\textstyle #2}\,\}}

\newcommand{\inv}{^{-1}}
\newcommand\conj{\overline}

\newcommand\lone{L^1}
\newcommand\ltwo{L^2}
\newcommand\linf{L^\infty}
\newcommand{\vn}{\mathit{VN}}
\newcommand{\VN}{\vn}

\newcommand{\G}{\mathbb{G}}

\newcommand{\ot}{\otimes}
\newcommand{\cop}{\Delta}
\newcommand{\id}{\mathrm{id}}
\newcommand{\M}{\mathit{M}}
\newcommand{\PD}{\mathit{P}}

\newcommand{\Com}{\Delta}

\newcommand{\MUG}{\M^u(\G)}
\newcommand{\MG}{\M(\G)}
\newcommand{\PG}{\PD(\G)}

\newcommand{\PUG}{\PD^u(\G)}

\newcommand{\QG}{\G}

\newcommand{\hQG}{\widehat{\QG}}

\newcommand{\wt}{\widetilde}

\newcommand\vnot{\mathop{\overline\otimes}}
\newcommand{\wot}{\vnot}

\newcommand{\Hil}{\mathsf{H}}

\newcommand{\Linf}{L^{\infty}}

\newcommand\dual[1]{\widehat{#1}}

\newcommand{\la}{\langle}
\newcommand{\ra}{\rangle}

\newenvironment{rlist}
{\begin{enumerate}[label=\textup{(\roman*)}, itemsep=1ex]}
{\end{enumerate}}

\begin{document}

\title[Convolution powers of contractive quantum measures]
{Fixed points and limits of convolution powers of contractive quantum measures}

\author{Matthias Neufang}
\address{School of Mathematics \& Statistics, Carleton University, Ottawa, ON,
	Canada K1S 5B6}
\email{mneufang@math.carleton.ca}

\author{Pekka Salmi}
\address{Department of Mathematical Sciences,
	University of Oulu, PL 3000, FI-90014 Oulun yliopisto, Finland}
\email{pekka.salmi@iki.fi}

\author{Adam Skalski}
\address{Institute of Mathematics of the Polish Academy of Sciences,
	ul.\'Sniadeckich 8, 00-656 Warszawa, Poland}
\email{a.skalski@impan.pl}

\author{Nico Spronk}
\address{Department of Pure Mathematics,
	University of Waterloo, Waterloo, ON, N2L 3G1, Canada}
\email{n.spronk@uwaterloo.ca}

\date{\today}


\begin{abstract}
	\noindent We study fixed points of contractive convolution operators associated to contractive quantum measures on locally compact quantum groups. We characterise the existence of non-zero fixed points respectively on $L^\infty(\G)$ and on $C_0(\G)$, and exploit these results to obtain for example the structure of the fixed points on the non-commutative $L_p$-spaces. Some consequences for the fixed points of classical convolution operators and Herz-Schur multipliers are also indicated.
\end{abstract}
\maketitle

Convolution operators associated to measures on a locally compact
group $G$ form a rich and interesting class of transformations, acting
on various  function spaces associated with $G$, and playing a key
role in the study of probabilistic, geometric and harmonic-analytic
phenomena related to $G$. In particular the random walk
interpretations provide motivation to analyse the limit behaviour of
iterations of convolution operators (as nicely described in \cite{Gren} and presented from the point of view of quantum generalisations in \cite{Pekkasurvey}), and thus in particular the nature
of idempotent probability measures. The latter are very well understood,
but, perhaps surprisingly, dropping the positivity requirement makes
the problem of characterising the class of idempotent measures very difficult, and it is only fully solved for
contractive measures  \cite{greenleaf:homo} and for abelian
groups \cite{Cohen}. More generally, one can ask about the structure
of the space of fixed points of a given convolution operator (see
\cite{Chu-Lau}, of which more will be said below), which at least in the positive case
can be interpreted as the collection of harmonic functions
for a given measure -- or as the space of functions on the
corresponding Poisson boundary (see for example \cite{Vadim} and references therein). 

If the group $G$ in question is abelian, the convolution operators can
be also viewed as Fourier multipliers. This perspective makes it
natural to study also fixed points of such operators even when $G$ is
non-abelian. This is to a large extent the point of view taken in the
lecture notes \cite{Chu-Lau}, where a lot of care is devoted to
studying for example the space of fixed points of contractive
Herz--Schur multipliers acting on the group von Neumann algebra
$\VN(G)$.  

The theory of locally compact quantum groups, as formulated by
Kustermans and Vaes in \cite{KV}, provides a framework which allows
asking these types of questions from a unified perspective, at the
same time vastly generalising the class of objects studied. The
concept of quantum convolution operators, as investigated for example
in \cite{jnr} and \cite{daws} has now been reasonably well understood,
and in particular several questions related to the nature of
idempotent states (\cite{salmi-skalski:idem} and references therein)
or the structure of the fixed points of positive quantum convolution
operators (\cite{knr}, \cite{knr2}) have found satisfactory
answers. In our previous work, \cite{NSSS}, we looked at the
idempotent contractive quantum measures and characterised their form,
generalising the results of Greenleaf mentioned above (see also
\cite{kasprzak}). Here we investigate the structure of the fixed point
spaces, so in other words `generalised harmonic elements', associated
to arbitrary quantum contractive measures.

The quantum context necessitates a distinction between the `universal'
and `reduced' quantum measures (respectively understood as bounded
functionals on the  universal and reduced algebra of continuous
functions on $\G$ vanishing at infinity); both of these induce
convolution operators on the algebra $L^\infty(\G)$ (as well as for
example on $C_0(\G)$). We work primarily in the more general setting
of universal measures. It is natural to expect that the
space of fixed points of a given operator is related to a limit of its
Ces\'aro averages. To be able to exploit this fully, we introduce yet
another class of generalised measures, i.e.\ the dual of the space of
`universal right uniformly continuous functions'. With this in hand, we
are able to show our first main result, Theorem
\ref{lemma:cesaro-nilpotent}, characterising the existence of non-zero
fixed points for the convolution operator on $L^\infty(\G)$. The
results on the structure of the fixed point space, often showing that
it is in fact one-dimensional (if the measure in question is
non-degenerate) require assuming more about the nature of a potential
fixed point. The central theorems of this type are Theorem
\ref{thm:group-like} and Proposition \ref{prop:more-equiv}. These results,
although at first glance quite technical in nature, turn out to have
several interesting consequences; both in the general quantum group
context, where they for example allow us to characterise fixed points
in non-commutative $L_p$-spaces, but also in the classical framework
and its dual, where they lead to generalisations of several earlier 
theorems (e.g.\ of \cite{Chu-Lau}). In several instances we are able
to connect the fixed points of the convolution by a quantum contractive
measure $\omega$ to the fixed points of the convolution by the absolute
value of $\omega$, which opens the way to exploiting the state results
already available in the literature. 

Finally we would like to recall that the fixed points of a positive (quantum) convolution operator admit a canonical  structure of a von Neumann algebra equipped with the action of the (quantum) group in question; this is nothing but a (quantum) Poisson boundary, as mentioned above. Once we consider general contractive (quantum) convolution operator, the fixed point space admits a canonical TRO structure, and the corresponding (quantum) group action. The study of such actions was initiated in \cite{PekkaAdamKyoto}; it is however fair to say that it remains in a relatively early stage and deeper harmonic analysis applications are only to be developed.

The detailed plan of the paper is as follows: in Section 1 we recall
basic facts and notations related to locally compact quantum groups
and introduce the algebra $\ruc^u(\G)^*$ and associated quantum
convolution operators, which play a useful technical role in what
follows. In Section 2 non-zero fixed points in $L^\infty(\G)$ are
studied, and their existence characterised. Here also we get the first
results on the structure of the fixed point space if the `universal'
convolution operator admits a non-zero fixed point in $\luc^u(\G)$. In
Section 3 we turn our attention to the preduals of
$L^\infty(\G)$-fixed points and use them to characterise certain
properties of the quantum group in question. The fourth section
concerns fixed points in $C_0(\G)$; the results obtained there are
used in the fifth section to discuss the existence of non-zero fixed
points in $L_p(\G)$ (for tracial Haar weights). In Section 6 the
general results are specialised to the context of classical locally
compact groups, and appropriately strengthened. Finally in Section 7
we discuss the `dual to classical' case.

\section{Preliminaries}

\subsection{Locally compact quantum groups}
We follow here the von Neumann algebraic approach to locally compact
quantum groups due to Kustermans and Vaes \cite{KV}, see also
\cite{knr} and \cite{knr2} for more background. A \emph{locally
  compact quantum group} $\QG$, effectively a virtual object, is
studied via the von Neumann algebra $L^{\infty}(\QG)$, playing the
role of the algebra of essentially bounded functions on $\QG$,
equipped with a \emph{comultiplication}
$\Com:L^{\infty}(\QG) \to L^{\infty}(\QG) \wot L^{\infty}(\QG)$, which
is a unital normal coassociative $*$-homomorphism. A locally compact
quantum group $\QG$ is by definition assumed to 
admit  a left Haar weight $\phi$ and a right Haar weight $\psi$ --
these are faithful, normal semifinite weights on $L^{\infty}(\QG)$
satisfying suitable invariance conditions. We can associate to $\QG$
also its \emph{algebra of continuous functions vanishing at infinity},
$C_0(\QG)$. It is a $C^*$-subalgebra of $L^{\infty}(\QG)$ and the
comultiplication restricts to a map from $C_0(\QG)$ to the multiplier algebra
of $C_0(\QG) \ot C_0(\QG)$.   We say that $\QG$ is \emph{compact} if  $C_0(\QG)$ is unital, in which case we denote it simply  $C(\QG)$.
The GNS representation space for the left Haar weight will be denoted
by $\ltwo(\QG)$. We may in fact assume that $C_0(\QG)$ is a
non-degenerate subalgebra of $B(\ltwo(\QG))$. Each locally compact
quantum group $\QG$ admits the \emph{dual} locally compact quantum
group $\hQG$; in fact the algebra $L^{\infty}(\hQG)$ acts naturally on
$\ltwo(\QG)$. We call $\QG$ \emph{discrete}, if $\hQG$ is compact. If $\QG=G$ happens to be a locally compact group, then
$L^{\infty}(\hQG)=\VN(G)$. Finally note that by analogy with the
classical situation we denote the predual of $L^{\infty}(\QG)$ by
$L^1(\QG)$. 

To each locally compact quantum group $\G$, we may also 
associate a universal $C^*$-algebra $C_0^u(\G)$ with comultiplication $\COP$;
see  \cite{ku}. 
There is a surjective $*$-homomorphism $\Lambda: C_0^u(\G)\to C_0(\G)$
such that $(\Lambda\ot\Lambda)\circ\COP = \cop\circ\Lambda$.
We call $\Lambda$ the \emph{reducing morphism} of $\G$. 
The reducing morphism of the dual $\dual\G$ is denoted
by $\dual\Lambda:C_0^u(\dual\G)\to C_0(\dual\G)$.

The  comultiplication $\cop$ of $C_0(\G)$ is implemented 
by the multiplicative unitary $\ww\in M(C_0(\G)\ot C_0(\dual \G))$:
\[
\cop(x) = \ww^*(1\ot x)\ww, \qquad x\in C_0(\G).
\]
The multiplicative unitary $\ww$ admits a universal lift
$\WW\in M(C_0^u(\G)\ot C_0^u(\dual \G))$ such that
$(\Lambda\ot \dual\Lambda)(\WW) = \ww$. 
We shall also need  half-universal versions of the 
bicharacter and the comultiplication: define
\[
\Ww = (\id\ot \dual\Lambda)\WW.
\]

The unitary antipode of $\G$ will be denoted by $R$. It is a
$^*$-antiautomorphism of $C_0(\G)$ (and in fact of $B(\ltwo(\QG))$)
of order 2. Its universal version, which is a   $^*$-antiautomorphism
of $C_0^u(\G)$ of order 2, will be denoted by $R_u$.
We have $\Lambda\circ R_u = R \circ \Lambda$.

We shall often write  $M^u(\G):= C_0^u(\G)^*$ and
$M(\G):= C_0(\G)^*$, and denote their unit balls by 
$M^u(\G)_1$ and $M(\G)_1$, respectively. 
Note that $M^u(\G)$ is a  completely contractive Banach algebra
with respect to convolution $\conv$, defined as the adjoint
of comultiplication $\COP$. Similarly also $M(\G)$ is a  completely
contractive Banach algebra.


\subsection{Convolution operators} 


Fix a locally compact quantum group $\QG$ and 
define
\[
\Cop(x) = \Ww^*(1\ot x)\Ww, \qquad x\in\linf(\G).
\]
Note that $\Ww$ as well as $\Cop(x)$ 
are in $M(C_0^u(\G)\ot \K(\ltwo(\G)))$ \cite{ku}.
Moreover, by equation (6.1) of \cite{ku} (applied to
$\Ww$), we have that
\begin{equation} \label{eq:reducing}
\Ww^*(1\ot \wt\Lambda(y))\Ww = (\id\ot \wt\Lambda)\circ\COP(y),
\qquad y\in C_0^u(\G)^{**},
\end{equation}
where $\wt\Lambda  : C_0^u(\G)^{**} \to \linf(\G)$ is the
normal extension of the reducing morphism $C_0^u(\G) \to C_0(\G)$.

For every $\omega\in M^u(\G)$ and $\phi\in \lone(\G)$,
there is a unique $\psi \in\lone(\G)$ such that
\[
(\omega\ot \phi\Lambda)\COP(a) = (\psi\Lambda)(a),\qquad a\in C_0^u(\G).
\]
We will write $\omega\conv \phi:= \psi$.
Similarly $\phi\conv \omega\in\lone(\G)$  is the unique $\psi' \in
\lone(\G)$ such that 
\[
( \phi\Lambda \ot \omega)\COP(a) = (\psi'\Lambda)(a),\qquad a\in C_0^u(\G).
\]
Hence we consider $\lone(\G)$ an ideal in $M^u(\G)$.
Note that for $x\in\linf(\G)$
\[
\omega\conv\phi (x) = \omega((\id\ot\phi)\Cop(x)).
\]
We will also need  the following formula:
\begin{equation}
  (\omega \star \phi)\circ R = (\phi \circ R) \star (\omega \circ
  R_u), \qquad \omega\in M^u(\G), \phi\in \lone(\G).
  \label{convantipode}
\end{equation} 
It follows from the above characterisations of the respective convolutions and the equality 
$ \chi(R_u \ot R_u)\circ \COP = \COP\circ R_u$ established in \cite{ku}.

Let
\[
\ruc^u(\G) = \clsp\set{(\id\ot \phi)\Cop(x)}{\phi\in\lone(\G), x\in\linf(\G)}.
\]
This is the universal version of $\ruc(\G)$ introduced and studied
in \cite{runde}; we have $\Lambda(\ruc^u(\G)) = \ruc(\G)$.

\begin{lem}
  \begin{rlist}
    \item $C_0^u(\G) = \clsp\set{(\id\ot
      \phi)\Cop(a)}{\phi\in\lone(\G), a\in C_0(\G)}$ 
    \item $\ruc^u(\G)$ is an an operator system
      \textup{(}i.e.\ a unital and self-adjoint linear
      closed subspace\textup{)} in $M(C_0^u(\G))$.
  \end{rlist}
\end{lem}

\begin{proof}
  (i)
  The elements of the form $(\id \ot \sigma)\ww$,
  with $\sigma\in  B(L^2(\G))_*$, are dense in $C_0(\G)$, 
  and similarly, by equation (5.2) of \cite{ku},
  the elements of the form $(\id \ot \sigma)\Ww$,
  with $\sigma\in  B(L^2(\G))_*$, are dense in $C_0^u(\G)$.
  By \eqref{eq:reducing} and Proposition~6.1 of \cite{ku}, we have
  \begin{equation} \label{eq:peneq}
  (\Cop\ot\id)(\ww) = (\id\ot\Lambda\ot \id)(\COP\ot\id)(\Ww) =
      \Ww_{13}\ww_{23}.
  \end{equation}
  For $\phi\in L^1(\G)$ and $\sigma\in  B(L^2(\G))_*$, we then have
  \[
  (\id\ot\phi)\bigl(\Cop((\id\ot\sigma)\ww)\bigr)
  =  (\id\ot\sigma)\bigl(\Ww(\phi\ot\id)(\ww)\bigr).
  \]
  Since the elements of the form $(\phi\ot\id)(\ww)$ are dense in
  $C_0(\dual \G)$ and $C_0(\dual \G)$ is non-degenerate on $\ltwo(\G)$,
  statement (i) follows from the previous equation and the discussion
  above.

  (ii)
  Since  $\Cop(x)$ is in $M(C_0^u(\G)\ot \K(\ltwo(\G)))$, it follows
  that $\ruc^u(\G)\sub M(C_0^u(\G))$. 
  That $\ruc^u(\G)$ is an operator system is obvious.
\end{proof}

\begin{remark}
In the case when the multiplicative unitary $\ww$ is semi-regular, one
can show that $\ruc^u(\G)$ is in fact a $C^*$-algebra: the argument given
in \cite[Theorem 5.6]{HNR2} works also in this case (with some obvious
modifications that require, for example, equation \eqref{eq:peneq}).
\end{remark}

Similarly, we can define $\luc^u(\G)$ starting from the right
multiplicative unitary $\vv\in M(C_0(\dual \G')\ot C_0(\G))$
and its half-universal lift $\vV\in M(C_0(\dual \G')\ot C_0^u(\G))$.
That is,
\[
\luc^u(\G) = \clsp\set{(\phi\ot\id)\coP(x)}{\phi\in\lone(\G), x\in\linf(\G)}.
\]
where
\[
\coP(x) = \vV(x\ot 1)\vV^*
\]
(note that the right multiplicative unitary $\vv$ implements
the comultiplication of $\G$ via $\cop(a) = \vv(a\ot 1)\vv^*$).
The results concerning $\ruc^u(\G)$ and its dual have
obvious analogues for $\luc^u(\G)$.

We may view $\M^u(\G)$ as a subspace of
$\ruc^u(\G)^*$, via strict extension.

Every $\rho\in\ruc^u(\G)^*$ defines a map
$\lt_\rho\col \linf(\G)\to \linf(\G)$ by
\[
\phi\bigl(\lt_\rho(x)\bigr) = \rho\bigl((\id\ot \phi)\Cop(x)\bigr)
\qquad (\phi\in\lone(\G)).
\]
The convolution operators on the universal level are defined in the
following proposition.

\begin{prop} \label{prop:convo-ops}
\begin{rlist}
\item For $\omega\in M^u(\G)$, the map
  $R_\omega^u: M(C_0^u(\G)) \to M(C_0^u(\G))$,
  \[
  R_\omega^u(x) = (\id\ot\omega)\COP(x),  \qquad x\in M(C_0^u(\G)),
  \]
  maps $\ruc^u(\G)$ to $\ruc^u(\G)$.
\item For $\rho\in \ruc^u(\G)^*$, the equation 
  \begin{equation} \label{eq:L-def}
  \omega(L_\rho^u(x)) = \rho(R_\omega^u(x)),  \qquad \omega\in
  M^u(\G), x\in\ruc^u(\G),
  \end{equation}
  defines an operator $L_\rho^u:\ruc^u(\G)\to \ruc^u(\G)$.
  In the case when $\rho\in M^u(\G)$, we have
  $L_\rho^u(x) =  (\rho\ot\id)\COP(x)$
  and $L_\rho^u$ maps $C_0^u(\G)$ to $C_0^u(\G)$. 
\end{rlist}

\end{prop}

\begin{proof}
  (i)
  For $\phi\in \lone(\G)$ and $y\in \linf(\G)$, we have
  \begin{align*}
  R_\omega^u((\id\ot\phi)\Cop(y)) &=
  (\id\ot \omega\ot \phi)(\COP\ot\id)\Cop(y) \\
   &= (\id\ot \omega\ot\phi)(\id\ot\Cop)\Cop(y) 
   =(\id\ot \omega\conv\phi)\Cop(y).
  \end{align*}
  Therefore $R_\omega^u$ maps $\ruc^u(\G)$ to $\ruc^u(\G)$.

  (ii)
  The equation \eqref{eq:L-def}
  defines an element $L_\rho^u(x)\in C_0^u(\G)^{**}$, so
  it is enough to check that in fact $L_\rho^u(x)\in \ruc^u(\G)$.
  To this end, let $x = (\id\ot\phi)\Cop(y)$ for 
$y\in \linf(\G)$ and $\phi\in \lone(\G)$. 
  Similarly as above, 
\begin{align*}
\omega(L_\rho^u(x)) &= \rho(R_\omega^u(x))
= \rho(R_\omega^u((\id\ot\phi)\Cop(y)))
= \rho((\id\ot\omega\conv\phi)\Cop(y))\\
&= \omega\conv\phi(L_\rho(y))
= (\omega\ot\phi)\Cop(L_\rho(y))
= \omega((\id\ot\phi)\Cop(L_\rho(y))).
\end{align*}
It follows that $L_\rho^u$ maps $\ruc^u(\G)$ to itself.
The second statement is obvious.
\end{proof}

Properties of $\ruc^u(\G)^*$ and the properties of the
convolution operators associated
with elements in $\ruc^u(\G)^*$ are collected in the following
proposition.

\begin{prop} \label{prop:collection}
\begin{rlist}
\item
  $\ruc^u(\G)^*$ is a Banach algebra under the convolution product
  \[
  \rho\conv \nu = \nu\circ L^u_\rho,
  \qquad \rho,\nu\in\ruc^u(\G)^*.  
  \]
\item $M^u(\G)$ is a closed subalgebra of $\ruc^u(\G)^*$. 
\item
  For every $\omega\in M^u(\G)$ and $\rho\in \ruc^u(\G)^*$,
  \[
  \rho\conv \omega  = \rho\circ R^u_\omega.
  \]
\item For every $\omega,\nu \in M^u(\G)$ and $\rho,\eta\in \ruc^u(\G)^*$,
  \[
  R_\omega^u \circ L_\rho^u = L_\rho^u\circ R_\omega^u, \qquad
  R_\omega^u \circ R_\nu^u = R^u_{\omega\conv \nu}, \qquad
  L_\eta^u\circ L_\rho^u = L_{\rho\conv \eta}^u.
  \]
    \item For every $\rho\in\ruc^u(\G)^*$,  
      \[
      \Lambda\circ \lt^u_\rho(x) = \lt_\rho \circ \Lambda(x),\qquad
      x\in\ruc^u(\G).
      \]
   \item For every $\rho, \eta\in\ruc^u(\G)^*$,  
     \[
     L_{\rho\conv \eta} = L_\eta\circ L_{\rho}. 
     \]
   \item For $\rho\in\ruc^u(\G)^*$, the map $L_{\rho}:\Linf(\G)\to \Linf(\G)$ is normal if and
     only if $\rho\in M^u(\G)$.
     \item Each of the maps $\omega\mapsto L_\omega$, $\omega \mapsto
       L^u_\omega$ and $\omega \mapsto R^u_\omega$
       \textup{(}the first two defined on $RUC^u(\G)^*$,
       the last one on $\MUG$\textup{)} are injective.
\end{rlist}
\end{prop}

\begin{proof}
(iv)
Here we check the first identity;
the second identity is easy to verify and the third identity
is proved in the next paragraph (because it needs the multiplication
on $\ruc^u(\G)^*$).
For $\omega, \nu\in M^u(\G)$, $\rho\in\ruc^u(\G)$ and
$x\in\ruc^u(\G)$, we have
\[
\nu(L_\rho^u\circ R_\omega^u(x)) =
\rho(R_\nu^u\circ R_\omega^u(x)) = \rho(R_{\nu\conv\omega}^u(x))
= \nu\conv\omega(L_\rho^u(x)) = \nu(R_\omega^u\circ L_\rho^u(x)).
\]

(i)
It follows from Proposition~\ref{prop:convo-ops} (ii) that
$\rho\conv\eta = \eta\circ L_\rho^u$ 
defines an element of $\ruc^u(\G)^*$
(note that $\ruc^u(\G)\subset M(C_0^u(\G))\subset C_0^u(\G)^{**}$).
For every $\rho,\eta, \nu \in \ruc^u(\G)$, we have by definition
\[
(\rho\conv \eta)\conv \nu = \nu\circ L_{\rho\conv \eta}^u. 
\]
Now for every $\omega \in M^u(\G)$  and $x\in\ruc^u(\G)$,
\[
\omega(L_{\rho\conv \eta}^u(x)) = \rho\conv \eta(R^u_\omega(x))
= \eta(L_\rho^u\circ R^u_\omega(x)) = \eta(R_\omega^u\circ L_\rho^u(x))
= \omega(L_\eta^u\circ L_\rho^u(x)).
\]
Hence $L_{\rho\conv \eta}^u = L_\eta^u\circ L_\rho^u$ and
\[
(\rho\conv \eta)\conv \nu = \nu\circ  L_\eta^u\circ L_\rho^u
= (\eta\conv \nu)\circ L_\rho^u = \rho\conv (\eta\conv \nu). 
\]
Therefore the convolution product on $\ruc^u(\G)^*$ is associative,
and clearly the norm is submultiplicative, so 
$\ruc^u(\G)^*$ is a Banach algebra.

(ii) is obvious and (iii) is immediate from
Proposition~\ref{prop:convo-ops} (ii). 
Also (v) is immediate and (vi) is similar to the analogous statement
    regarding  $L_\rho^u$ in (iv). Condition (viii) is easy to check directly from the definitions.

(vii)
If $\rho\in M^u(\G)$, then for every $\phi\in\lone(\G)$
and $x\in\linf(\G)$,
\[
\phi\bigl(\lt_\rho(x)\bigr) = \rho\bigl((\id\ot \phi)\Cop(x)\bigr)
= \rho\conv \phi(x).
\]
As $\rho\conv \phi\in \lone(\G)$, it follows that $L_\rho$ is
normal.

It remains to show that $L_\rho$ cannot be normal
if $\rho\in C_0^u(\G)^\perp\sm\{0\}$.
Now for every $a\in C_0^u(\G)$, we have
$L_\rho(\Lambda(a)) = \Lambda(L_\rho^u(a))$ and 
as noted in the proof of Corollary~\ref{cor:decomp},
$L_\rho^u(a) = 0$. Hence $L_\rho = 0$ on $C_0(\G)$.
If $L_\rho$ is normal, then $L_\rho = 0$ because $C_0(\G)$ is weak*-dense in
$\linf(\G)$, and consequently $\rho = 0$ by (viii). 
\end{proof}

  Sometimes we need to work on $C_0^u(\G)^{**}$ in which case we may
  consider normal maps $\wt{R_\omega^u}:C_0^u(\G)^{**}\to C_0^u(\G)^{**}$
  and $\wt{L_\omega^u}:C_0^u(\G)^{**}\to C_0^u(\G)^{**}$, given
  $\omega\in M^u(\G)$. These are defined by 
  \[
  \wt{R_\omega^u}(x) = (\id\ot \omega)\wt\COP(x)\quad\text{and}\quad
  \wt{L_\omega^u}(x) = (\omega\ot\id)\wt\COP(x).
  \]
  Note that  on $M(C_0^u(\G))$ the maps $\wt{R_\omega^u}$ and
  $\wt{L_\omega^u}$ agree  with the maps $R_\omega^u$ and
  $L_\omega^u$, respectively.

We will need one more result, essentially contained in Proposition 2.1 of \cite{BBMS}; here we provide a different proof.

\begin{prop}
  Let $A$ be a $C^*$-algebra. Then
  \[
  M(A)^* = A^*\oplus_1 A^\perp
  \]
  where as usual $A^*\subset M(A)^*$ is defined using strict extensions
  of functionals in $A^*$.
\end{prop}

\begin{proof}
  For every $\nu\in A^*$, denote its strict extension in $M(A)^*$
  by $\tilde \nu$. For every $\mu\in M(A)^*$ define $\mu_0 = \mu|_A$
  and $\mu_1 = \mu - \tilde\mu_0$. Then
  \[
  \mu = \tilde\mu_0 + \mu_1
  \]
  will be the required decomposition.

  Suppose first that $\mu$ is a positive functional.
  Then clearly also $\tilde\mu_0$ is positive.
  We claim that $\tilde\mu_0\le \mu$ so that also
  $\mu_1$ is positive. Let $x\ge 0$ in $M(A)$.
  If $(e_i)_{i \in \Ind}$ is an increasing, positive contractive approximate
  identity in $A$, then $(x^{1/2}e_i x^{1/2})_{i \in \Ind}$ is an increasing
  net converging strictly to $x$. Hence
  \[
  \tilde\mu_0(x) = \lim_{i \in \Ind} \mu_0(x^{1/2}e_i x^{1/2})
  = \lim_{i \in \Ind} \mu(x^{1/2}e_i x^{1/2}) \le \mu(x).
  \]

  So $\mu = \tilde\mu_0 + \mu_1$ is a decomposition of $\mu$ into
  positive components and so
  \[
  \|\mu\| = \mu(1) = \tilde\mu_0(1) + \mu_1(1) =
  \|\tilde\mu_0\| + \|\mu_1\|.
  \]
  Therefore the claim holds for positive $\mu$.

  Let now $\mu\in M(A)^*$ be arbitrary. Then $\mu$ has a
  (right) polar decomposition
  \[
  \mu = \nu.u
  \]
  where $\nu\in M(A)^*$ is positive and $u\in M(A)^{**}$
  is a partial isometry
  (the polar decomposition follows by considering $M(A)^*$ as the
  predual of the von Neumann algebra $M(A)^{**}$).

  We claim that $\tilde\nu_0.u = \tilde\mu_0$. 
  These functionals coincide on $A$ so it is enough to
  show that $\tilde\nu_0.u$ is strictly continuous.
  As there are two different extensions at play
  (from $A$ to $M(A)$ and from $M(A)$ to $M(A)^{**}$),
  it is clearer to use angle brackets to denote the duality
  of $M(A)^*$ and $M(A)^{**}$.
  Let $\nu_0 = a.\eta$ be a factorisation of $\nu_0$ on $A$
  ($\eta\in A^*$ and $a\in A$).
  The functional $\tilde\nu_0$ is defined by
  \[
  \tilde\nu_0(y) = \eta(ya), \qquad y\in M(A),
  \]
  so that by a weak*-approximation we have also $\tilde\nu_0(z) = \tilde\eta(za)$ for all $z \in M(A)^{**}$, where $\tilde\eta$ denotes the relevant weak*-extension.
  Let $(x_i)_{i \in \Ind}\subset M(A)$ converge strictly to $x$ in $M(A)$.
  Now
  \begin{align*}
  \la \tilde\nu_0, ux \ra &= \la \tilde\eta, uxa \ra
  = \la \tilde\eta, u(\lim_{i \in \Ind} x_i a) \ra
  = \lim_{i \in \Ind} \la \tilde\eta, u x_i a \ra
  = \lim_{i \in \Ind} \la \tilde\nu_0, u x_i \ra,
  \end{align*}
  which shows that $\tilde\nu_0.u$ is strictly continuous,
  as required.

  As $\tilde\nu_0.u = \tilde\mu_0$, it follows from
  the uniqueness of the decomposition that
  $\nu_1.u = \mu_1$.
  Then we have
  \begin{align*}
    \|\mu\| \le   \|\tilde\mu_0\| + \|\mu_1\|
    =   \|\tilde\nu_0.u\| + \|\nu_1.u\|
    \le \|\tilde\nu_0\| + \|\nu_1\|
    = \|\nu\|  = \|\mu\|,
  \end{align*}
  where we used the first part of the proof (applied to the positive
  functional $\nu$) and the fact that $\|\mu\| = \|\nu\|$.
\end{proof}

\begin{cor} \label{cor:decomp}
We have a decomposition 
\[
\ruc^u(\G)^* = M^u(\G) \oplus_1 C_0^u(\G)^{\perp}
\]
where $M^u(\G)$ is a subalgebra of $\ruc^u(\G)^*$
and  $C_0^u(\G)^{\perp}$ is a weak*-closed ideal in $\ruc^u(\G)^*$.
\end{cor}

\begin{proof}
  It remains to show that $C_0^u(\G)^{\perp}$ is an ideal.
  Suppose first that  $\mu,\nu\in\ruc^u(\G)^*$
  with $\mu\in C_0^u(\G)^{\perp}$. For every $\omega\in \ruc^u(\G)^*$
  and $a\in C_0^u(\G)$,
  we have
  \[
  \omega(L_\mu^u(a)) = \mu(R_\omega^u(a)) = 0 
  \]
  as $R_\omega^u(a) \in C_0^u(\G)$, and so $L_\mu^u(a) = 0$. It
  follows that 
  $\mu\conv \nu = \nu\circ L_\mu^u\in C_0^u(\G)^\perp$.

  It remains to check the case when $\mu\in\ruc^u(\G)^*$ and
  $\nu\in C_0^u(\G)^\perp$. But now
  $L_\mu^u(a)\in C_0^u(\G)$ whenever $a\in C_0^u(\G)$,
  and so $\mu\conv\nu \in C_0^u(\G)^\perp$.
\end{proof}

\begin{lem}  \label{idempalternative}
If $\omega \in \ruc^u(\G)^*$ is a contractive idempotent, then
either $\omega \in M^u(\G)$ or $\omega \in C_0^u(\G)^{\perp}$.
\end{lem}

\begin{proof}
  Write $\omega = \omega_0 + \omega_1$ where $\omega_0\in M^u(\G)$
  and $\omega_1\in C_0^u(\G)^\perp$. 
  Then
  \[
  \omega_0 + \omega_1 = \omega = \omega\conv\omega =
  \omega_0 \conv\omega_0 + \omega_0 \conv \omega_1
  + \omega_1 \conv \omega_0 + \omega_1\conv\omega_1.
  \]
  Since $C_0(\G)^{\perp}$ is an ideal, it follows that
  $\omega_0 \conv\omega_0 = \omega_0$.
  Now $\|\omega\| = \|\omega_0\| + \|\omega_1\|$ by
  Corollary~\ref{cor:decomp}, so $\omega_0$ is a contractive idempotent. 
  Hence $\|\omega_0\|$ is either $0$ or $1$. If $\|\omega_0\| = 0$,
  then $\omega = \omega_1\in  C_0^u(\G)^{\perp}$.
  If $\|\omega_0\| = 1$, then $\|\omega_1\| = 0$ and
  $\omega = \omega_0 \in M^u(\G)$.
\end{proof}


\section{Fixed points in $\linf(\G)$ }


Let $\G$ be a locally compact quantum group.
For a map $L\col \linf(\G)\to \linf(\G)$, write
\[
\Fix L = \set{x\in \linf(\G)}{L(x) = x}.
\]
We are mostly interested in $\Fix L_\omega$ with $\omega\in M^u(\G)_1$,
in which case $\Fix L_\omega$ is weak*-closed. 
Note that if $\|\omega\|<1$, then $\Fix L_\omega = \{0\}$,
so we may concentrate on the case $\|\omega\|=1$.

Given $\omega \in M^u(\QG)_1$ and $n \in \N$ we write
\[
S_n(\omega) = \frac{1}{n} \sum_{k=1}^{n} \omega^{\conv k}\in M^u(\G)_1.
\]
Fix a free ultrafilter $p$ on $\N$ and take the weak$^*$ limit of the
sequence $S_n(\omega)$ along $p$ in $\ruc^u(\G)^*$. We will denote
this limit by $\wt{\omega}_{p}$ (later we will simply write
$\wt\omega$ once $p$ is fixed).

\begin{lem}\label{statelimits}
Let $\omega \in M^u(\G)_1$ and let $p$ be a free ultrafilter.
Then
\[
\wt{\omega}_{p} = \omega\conv\wt{\omega}_{p}
= \wt{\omega}_{p}\conv \omega
= \wt{\omega}_{p}\conv \wt{\omega}_{p}.
\]
In particular $\wt{\omega}_{p}$ is either 0 or a contractive
idempotent. Finally, the following conditions are also equivalent:
\begin{rlist}
\item $\wt{\omega}_{p}$ is a state for some free ultrafilter $p$;
\item $\wt{\omega}_{p}$ is a state for every free ultrafilter $p$;
\item $\omega$ is a state.
\end{rlist}
\end{lem}

\begin{proof}
The first statement follows from the identity
\begin{equation} \label{eq:absorb}
  S_n(\omega)\conv \omega =
  \omega\conv S_n(\omega) = S_n(\omega) - \frac1n\omega + \frac1n\omega^{n+1}.
\end{equation}

The implications (iii)$\implies$(ii)$\implies$ (i) are obvious.
Suppose that $\wt{\omega}_{p}$ is a state.
Then 
\[
1 = \wt{\omega}_{p}(1) =
\omega\conv \wt{\omega}_{p} (1) = \omega(1)\wt{\omega}_{p}(1)=
\omega(1)
\]
so (iii) follows.
\end{proof}

\begin{lem} \label{lem:fix_L}
  Let $\omega \in M^u(\G)_1$ and let $p$ be a free ultrafilter.
  The map $L_{\wt{\omega}_p}$ is a
  projection onto $\Fix L_{\omega} = \Fix L_{\wt{\omega}_{p}}$.  
\end{lem}

\begin{proof}
To see that $L_{\wt{\omega}_{p}}$ and $L_\omega$ have the same
fixed points, 
note first that for all $\phi\in\lone(\G)$ and $x\in\linf(\G)$
\[
\phi\bigl(L_{\wt{\omega}_{p}}(x)\bigr) = 
\wt{\omega}_{p}\bigl((\id\ot\phi)\Cop(x)\bigr) = 
\plim S_n(\omega)\bigl((\id\ot\phi)\Cop(x)\bigr).
\]
It is immediate that $\Fix L_{\omega} \subset \Fix
L_{\wt{\omega}_{p}}$. On the other hand if $x$ is a fixed point 
of $L_{\wt{\omega}_{p}}$, then it follows from the above
identity that 
\begin{align*}
\phi\bigl(L_{\omega}(x)\bigr) 
&=\phi\bigl(L_\omega(L_{\wt{\omega}_{p}}(x))\bigr)
= \phi(L_{\wt\omega_p\conv\omega}(x))
=\wt\omega_p\conv\omega\bigl((\id\ot\phi)\Cop(x)\bigr) \\
&= \plim
 S_n(\omega)\bigl((\id\ot(\omega\conv\phi))\Cop(x)\bigr)
= \plim \bigl(S_n(\omega)\conv
\omega\bigr)\bigl((\id\ot\phi)\Cop(x)\bigr) \\
&= \phi\bigl(L_{\wt\omega_{p}}(x)\bigr) = \phi(x) 
\end{align*}
by \eqref{eq:absorb}.

Since $\wt\omega_p$ is either a non-zero contractive idempotent or $0$,
it follows that $L_{\wt\omega_p}$ is an idempotent whose image is $\Fix L_{\wt\omega_p}$.
\end{proof}

The following theorem characterises the case when there
are no non-zero fixed points. Note that convolution operators determined by
\emph{states} always have fixed points (namely, the constants)
but this need not be the case with contractive functionals.

\begin{tw} \label{lemma:cesaro-nilpotent}
Let $\omega\in M^u(\G)_1$. The following are equivalent:
\begin{rlist}
\item $\lt_\omega$ has no non-zero fixed points in $\linf(\G)$;
\item $\lt_\omega$ has no non-zero fixed points in $\ruc(\G)$;
\item $\lt_\omega^u$ has no non-zero fixed points in $\ruc^u(\G)$;
\item $S_n(\omega)\to 0$ weak* in $\ruc^u(\G)^*$;
\item $\wt{\omega}_p= 0$ for all free ultrafilters $p$;
\item $\wt{\omega}_p= 0$ for some free ultrafilter $p$.
\end{rlist}
\end{tw}

\begin{proof}
(iii) $\implies$ (i): Let $x$ be a non-zero fixed point of $\lt_\omega$
in $\linf(\G)$ and pick $\phi\in\lone(\G)$ such that
$\tilde x:= (\id\ot\phi)\Cop(x)\ne 0$.
Then $\tilde x\in\ruc^u(\G)$ and for every $\mu\in\ruc^u(\G)^*$
\begin{align*}
\mu\bigl(\lt_\omega^u(\tilde x)\bigr)
&= \omega\conv \mu\bigl( (\id\ot \phi)\Cop(x) \bigr)
= \phi\bigl( L_{\omega\conv \mu}(x) \bigr)
= \phi\bigl( L_\mu\circ L_{\omega}(x) \bigr)
= \phi\bigl( L_\mu(x) \bigr)
= \mu(\tilde x).
\end{align*}
So $\tilde x$ is a non-zero fixed point of $\lt_\omega^u$ in $\ruc^u(\G)$.

(ii) $\implies$ (i) is similar to (iii) $\implies$ (i).
The converse implication (i) $\implies$ (ii) is trivial.

The implications  (iv) $\iff$ (v) $\implies$ (vi) are obvious.

(i) $\implies$ (v):
If $\wt{\omega}_p\ne 0$, there is $x\in \linf(\G)$ such that
$\lt_{\wt{\omega}_p}(x)\ne 0$. Moreover,
$\lt_\omega\bigl(\lt_{\wt{\omega}_p}(x)) = \lt_{\wt{\omega}_p}(x)$.

(vi) $\implies$ (iii):
Suppose that there exist a non-zero fixed point $x$ of
$\lt_\omega^u$ in $\ruc^u(\G)$
and pick $\nu\in M^u(\G)$ such that $\nu(x)\ne 0$
(recall that $\ruc^u(\G)\subset M(C_0^u(\G))\subset C^u_0(\G)^{**}$).
Then using Proposition~\ref{prop:collection}, we have
\begin{align*}  
\nu\bigl(\lt_{\wt\omega_p}^u(x)\bigr)
&=\wt\omega_p\conv \nu(x)
= \wt\omega_p\bigl(R_\nu^u(x))
=\plim S_n(\omega)\bigl(R_\nu^u(x)\bigr)\\
&=\plim \frac 1n \sum_{k=1}^n \nu(\lt_{\omega^{\conv k}}^u(x))
= \nu(x).
\end{align*}
It follows that $\wt\omega_p\ne 0$.
\end{proof}

The \emph{right absolute value} of $\omega\in M^u(\G)$ is
a positive functional $|\omega|$ in $M^u(\G)$ defined through
the polar decomposition $\omega = |\omega|(\cdot u)$  where
$u\in C_0^u(\G)^{**}$ is a partial isometry satisfying some extra properties;
see Definition III.4.3 of \cite{takesaki:vol1}. 
Then the \emph{left absolute value} $|\omega|_\ell$ is
defined by $\omega = |\omega|_\ell(u^*\cdot )$.

\begin{lem}\label{lem:abs}
  Let $M$ be a von Neumann algebra and $\nu\in M_*$ with $\|\nu\| = 1$.
  Suppose that $x\in M$ is such that $\|x\|=1$ and $\nu(x^*) = 1$.
  Then $\nu = x.|\nu|$ and $|\nu| = x^*.\nu$.
\end{lem}

\begin{proof}
  By the proof of Theorem III.4.2 of \cite{takesaki:vol1}, $\nu = u.|\nu|$
  where $x = u|x|$ is the polar decomposition of $x$.
  Note that $|\nu|(|x|) = \nu(|x|u^*) = 1 $.
  By the Cauchy--Schwarz inequality
  \[
  |\nu|(|x|)^2 \le |\nu|(|x|^2)|\nu|(1) \le 1 = |\nu|(|x|)^2. 
  \]
  Hence $|x|$ is in the multiplicative domain of $|\nu|$, and so
  for every $y\in M$
  \[
  x.|\nu|(y) = |\nu|(y u |x|) = |\nu|(yu)|\nu|(|x|) = \nu(y). 
  \]
  As $|\nu| = u^*.\nu$, we also have
  \[
  x^*.\nu(y) = \nu(y |x| u^*) = |\nu|(y|x|) = |\nu|(y). 
  \]
\end{proof}

The above general lemma allows us to say something about fixed points of contractive convolution operators.

\begin{lem} \label{lemma:group-like}
Suppose that $\omega\in M^u(\G)$ is contractive and
$v\in C_0^u(\G)^{**}$ is such that $v^*$ is a fixed point of
$\wt{L_\omega^u}$, $\|v\| = 1$  and $\omega(v^*) = 1$. 
Then 
\[
\wt\COP^{(k)}(v) - v^{\ot(k+1)} \in N_{|\omega|^{\ot(k+1)}} \qquad \text{and} \qquad
\omega^{\conv k} = v.|\omega|^{\conv k}
\]
for every $k \in \bn$. Moreover,
$|\omega|^{\conv k}(av^* v) = |\omega|^{\conv k}(v^* v a)= |\omega|^{\conv k}(a)$
for every $a\in C_0^u(\G)^{**}$ and $k \in \bn$.
\end{lem}

\begin{proof}
As $\|v\| = 1$ and $\omega(v^*)=1$, it follows from
Lemma~\ref{lem:abs} that $\omega = v.|\omega|$.  Moreover,
$|\omega|(v^*v) = 1$. 

Fix $k \in \bn$. Since $v^*$ is a fixed point of $\wt{\lt_\omega^u}$, we have
\[
1 = \omega(v^*) = \omega^{\ot (k + 1)}\bigl(\wt\COP^{(k)}(v^*)\bigr)
= |\omega|^{\ot( k + 1)}\bigl(\wt\COP^{(k)}(v^*)v^{\ot (k+1)}\bigr).
\]
Therefore
\begin{align*}
&|\omega|^{\ot(k +1)}\bigl((\wt\COP^{(k)}(v)-v^{\ot (k + 1)})^*
  (\wt\COP^{(k)}(v)-v^{\ot( k + 1)})\bigr)\\
&\qquad= |\omega|^{\ot( k +1)}\bigl(\wt\COP^{(k)}(v^* v) - \wt\COP^{(k)}(v^*)v^{\ot(k + 1)}
  -  (v^*)^{\ot( k + 1)}\wt\COP^{(k)}(v) + (v^* v)^{\ot( k + 1)}) \bigr) \\
&\qquad= |\omega|^{\ot (k +1)}\bigl(\wt\COP^{(k)}(v^* v)\bigr) - 1 -1 + 1 = 0.
\end{align*}
So we have
\[
\wt\COP^{(k)}(v) - v^{\ot (k+1)} \in N_{|\omega|^{\ot (k+1)}},
\]
and it is then immediate that $\omega^{\conv k} = v.|\omega|^{\conv k}$.

As a by-product, we get that $|\omega|^{\conv k}(v^*v) = 1$, and so
by the Cauchy--Schwarz inequality
\[
1 = |\omega|^{\conv k}(v^*v)^2 \le |\omega|^{\conv k}(v^*vv^*v)
\le \|v^*v v^* v\| \le 1.
\]
By Choi's theorem on multiplicative domains,
$v^* v$ is in the multiplicative domain of $|\omega|^{\conv k}$,
and so $|\omega|^{\conv k}(av^* v) = |\omega|^{\conv k}(v^* v a)
= |\omega|^{\conv k}(a)$.
\end{proof}

In the case when $\omega$ is a state, the next lemma trivialises (one can take $v=1$).

\begin{lem} \label{lemma:v-group-like}
Let $\omega\in M^u(\G)_1$.
If $\lt_\omega$ has a non-zero fixed point in $\linf(\QG)$,
then there exists $v\in C_0^u(\G)^{**}$ such that
\[
\wt\COP^{(k)}(v) - v^{\ot k+1} \in N_{|\omega|^{\ot k+1}} \qquad \text{and} \qquad
\omega^{\conv k} = v.|\omega|^{\conv k}
\]
for every $k \in \bn$.
Moreover, $|\omega|^{\conv k}(av^* v) = |\omega|^{\conv k}(v^* v a) = |\omega|^{\conv k}(a)$
for every $a\in C_0^u(\G)$ and $k \in \bn$.
\end{lem}

\begin{proof}
Let $x\in\linf(\G)$ be a non-zero fixed point of $\lt_\omega$.
Pick a sequence $(\phi_n)_{n \in \bn}\sub \ball(\lone(\G))$ such that
$\phi_n(x) > \|x\| - 1/n$ for each $n \in \bn$.
Define
\[
v_n = \frac{(\id\ot \phi_n^*)\Cop(x^*)}{\phi_n^*(x^*)}\in \ruc^u(\G)
\]
and let $v$ be a weak* cluster point of $(v_n)$ in $C_0^u(\G)^{**}$.
Since $L_\omega(x) = x$, it follows that $\omega(v_n^*) = 1$
for every $n\in \bn$ and so $\omega(v^*) = 1$.
Moreover, since $\phi_n(x) \to \|x\|$, it follows
that $\|v\|\le 1$, and so $\|v\|=1$.
Finally, we note that $v^*$ is a fixed point of $\wt{L_\omega^u}$: indeed,
it is easy to calculate that $L_\omega^u(v_n^*) = v_n^*$ and
the rest follows as $\wt{L_\omega^u}$ is normal when $\omega\in M^u(\G)$.
Now we may apply Lemma~\ref{lemma:group-like} to obtain the result.
\end{proof}

In the case when $L_\omega^u$ has a fixed point in $\luc^u(\G)$
we obtain a tighter result. Note that by
Lemma~\ref{lemma:cesaro-nilpotent} a fixed point in $\linf(\G)$
implies a fixed point in $\ruc^u(\G)$ but not neccessarily in $\luc^u(\G)$.

We call a contractive functional $\omega\in \M^u(\G)$
\emph{non-degenerate} if both $|\omega|$ and $|\omega|_\ell$ are
non-degenerate in the sense of
\cite{knr}, i.e.\ for each non-zero $x\in C_0^u(\QG)_+$ there exists
$n \in \N$ such that $|\omega|^{\conv n}(x) >0$ and
$|\omega|_\ell^{\conv n}(x)>0$. We use the analogous 
notion of degeneracy for $\omega\in \MG$.

\begin{theorem} \label{thm:group-like}
  Suppose that $\omega\in M^u(\G)$ is non-degenerate and 
  contractive. If $L_\omega^u$ has a non-zero fixed point in $\luc^u(\G)$,
  then there is a unitary $v\in \luc^u(\G)$ (even $\uc^u(\G)$)
  such that
  \[
  \COP(v) = v\ot v,
  \]
  and $\Fix L_\omega^u = (\Fix L_{|\omega|}^u) v^*$.
  Taking $u = \Lambda(v)$ we have
  $\Fix L_\omega = (\Fix L_{|\omega|}) u^*$.
  \end{theorem}

\begin{proof}
  Let $x\in\luc^u(\G)$ be a fixed point of $L_\omega^u$ such that $\|x\| = 1$. 
  Fix $\mu\in\luc^u(\G)^*$ such that $\|\mu\|= 1$ and $\mu(x) = 1$.
  Define $v = R_\mu^u(x)^*\in \luc^u(\G)$ (where $R_\mu^u$ is defined
  analogously to $L_\mu^u$, introduced in Proposition~\ref{prop:collection}). 
  Then
  \[
  \omega(v^*) = \omega(R_\mu^u(x)) = \mu(L_\omega^u(x)) = \mu(x) = 1,
  \]
  and in particular $\|v\|=1$. Moreover, $v^*$ is a fixed
  point of $L_\omega^u$ as $L_\omega^u$ commutes with $R_\mu^u$.
  Now we may apply Lemma~\ref{lemma:group-like}.
  
  We claim that $\COP(v) = v\ot v$.
  For every $k,l\in \N$ we have
  \begin{align*}
    &|\omega|^{\conv k}\ot |\omega|^{\conv l}
         \bigl((\COP(v)-v\ot v)^*(\COP(v)-v\ot v)\bigr)\\
    &\qquad= |\omega|^{\conv k}\ot |\omega|^{\conv l}
         \bigl((\COP(v^*v)-\COP(v^*)(v\ot v)
         -(v^*\ot v^*)\COP(v) + v^*v\ot v^* v\bigr)\\
    &\qquad=|\omega|^{\conv(k+l)}(v^*v) - \omega^{\conv(k+l)}(v^*) -
         \conj{\omega^{\conv(k+l)}(v^*)} +
      |\omega|^{\conv k}(v^*v) |\omega|^{\conv l}(v^*v)  = 0.
  \end{align*}
  Now   $X = (\COP(v)-v\ot v)^*(\COP(v)-v\ot v)\in M(C_0^u(\G)\ot C_0^u(\G))$
  is such that $X\ge 0$ and  $|\omega|^{\conv k}\ot |\omega|^{\conv l}(X) = 0$.
  Since for every $l\in \N$ 
  \[
  |\omega|^{\conv l}((|\omega|^{\conv k}\ot \id)(X)) = 0
  \]
  it follows by non-degeneracy that $(|\omega|^{\conv k}\ot \id)(X) = 0$.
  Then, for every $\sigma\ge 0$ in $C_0^u(\G)^*$,
  we have $(\id\ot\sigma)(X)\ge 0$ and 
  $|\omega|^{\conv k}((\id\ot\sigma)(X)) = 0$ for every $k\in\N$.
  Therefore $(\id\ot\sigma)(X) = 0$ for every $\sigma\ge 0$
  and so $X = 0$. 
  Consequently $\COP(v) = v\ot v$ as required.

  Now for every $a\in C_0^u(\G)$ we have
  \[
  |\omega|^{\conv k}\bigl((av^*v - a)^*(av^*v - a)\bigr)
   = |\omega|^{\conv k}\bigl(v^*v a^*av^*v - v^*v a^*a - a^*av^*v +
       a^*a\bigr)
   = 0.
  \]
  Non-degeneracy implies that $av^* v = a$, and then
  it follows that $v^*v = 1$.
  Similarly $vv^* = 1$ follows
  by applying non-degeneracy of the left-hand sided absolute value
  $|\omega|_\ell$.
  
  If $y\in \Fix L_{|\omega|}^u$, then  
  \[
  L_\omega^u(yv^*) = (\omega\ot\id)\COP(yv^*) =
  (\omega\ot\id)(\COP(y)(v^*\ot v^*)) =
   (|\omega|\ot\id)(\COP(y)) v^* = yv^*.
  \]
  Hence $\Fix L_{|\omega|}^u v^* \subset \Fix L_{\omega}^u$.
  Conversely, if $z\in \Fix L_{\omega}^u$, then
  \[
  L_{|\omega|}^u(zv) = (|\omega|\ot\id)\COP(zv) =
  (|\omega|\ot\id)(\COP(z)(v\ot v)) 
  = (\omega\ot\id)(\COP(z)) v = zv.
  \]
  Therefore $\Fix L_{\omega}^u = (\Fix L_{|\omega|}^u) v^*$
  and $\Fix L_{|\omega|}^u = (\Fix L_{\omega}^u) v $.

  Finally, $u = \Lambda(v)$ is a group-like unitary
  in $M(C_0(\G))$ and $u^*$ is a fixed point of $L_\omega$.
  Similarly as above
  $\Fix L_{\omega} = (\Fix L_{|\omega|}) u^*$
  and $\Fix L_{|\omega|} = (\Fix L_{\omega}) u$.
  \end{proof}

We say that a locally compact quantum group $\G$ is
\emph{ universally SIN} if $\luc^u(\G) = \ruc^u(\G)$. Note that all discrete
and compact quantum groups are universally SIN as are duals of locally compact
groups. Recall that Hu, Neufang and Ruan called a locally compact
quantum group SIN (from `having small invariant neighbourhoods') if
$\luc(\G) = \ruc(\G)$. If $\G$ is universally SIN, it is SIN, and the
two notions trivially coincide for coamenable $\G$. We do not know if
they coincide in general.

A weak*-closed subspace of a von Neumann algebra is called a
\emph{$W^*$-sub-TRO} if it is closed under the ternary product
$(x, y, z)\mapsto xy^*z$. 

\begin{cor}
  Suppose that $\G$ is a locally compact quantum group that is universally SIN,
  and let $\omega \in \M^u(\G)_1$ be non-degenerate.
  If  $\Fix L_{\omega}$ is a $W^*$-sub-TRO of $\Linf(\G)$,
  then $\dim (\Fix L_{\omega}) \leq 1$ 
  \textup{(}moreover if $\Fix L_{\omega}$ is non-zero, it
  contains a unitary\textup{)}.
\end{cor}

\begin{proof}
The case when $\Fix L_{\omega} = \{0\}$ is clear so we may suppose
that $\Fix L_{\omega} \ne \{0\}$. We then first use Theorem \ref{lemma:cesaro-nilpotent} and the assumption that $\QG$ is universally SIN to deduce that the assumptions of  Theorem~\ref{thm:group-like} hold. Thus there is a group-like unitary
$u$ such that   $\Fix L_\omega = (\Fix L_{|\omega|}) u^*$.
If $\Fix L_\omega$ is a $W^*$-sub-TRO, then $\Fix L_{|\omega|}$ is a
von Neumann subalgebra. Hence $\Fix L_{|\omega|}\cong \complex$
by \cite[Theorem 3.6]{knr}.
\end{proof}

We do not know if the universally SIN assumption of the previous result is
actually necessary; our methods seem however to be limited to this case.


\section{Annihilators of $\Fix L_{\omega}$ in $L^1(\QG)$} 


In this short section we describe properties of the pre-annihilators of
fixed point spaces of the type $\Fix L_\omega\sub \linf(\G)$, i.e.\ certain
one-sided ideals in $L^1(\QG)$.

\begin{deft} \label{def:preann}
  Given $\omega \in \PUG$, define a closed right ideal $I_\omega$ in
  $L^1(\QG)$ by the formula 
  \[
  I_{\omega} := \textup{cl } \{f - \omega \star f: f\in L^1(\QG) \}.
  \]
Moreover, let $L^1_0(\QG)$ denote the augmentation ideal in
   $L^1(\QG)$, i.e.\ $L^1_0(\QG)= \{f\in L^1(\QG): f(1) = 0\}$. 
\end{deft}

\begin{lem} \label{preann}
	Let $\omega \in \PUG$. Then the right ideal $I_\omega$ is equal to
        the pre-annihilator of $\Fix L_\omega$. In particular
        $I_\omega \subset L^1_0(\QG)$. 
\end{lem}
\begin{proof}
	To show the first statement it suffices to consider the following chain of equivalences:	
	\[ x \in I_{\omega}^{\perp} \Longleftrightarrow \forall_{f \in L^1(\QG)}\, f(x) - \omega \star f(x) =0  \Longleftrightarrow \forall_{f \in L^1(\QG)}\, f(x) = f(L_{\omega}(x)) \Longleftrightarrow x = L_{\omega}(x). \]
	The second statement is now obvious.
\end{proof}

The following result can be proved exactly as its classical counterpart,  \cite[Theorem 1.2]{Willis}. We outline the argument for completeness.

\begin{prop}
	Assume that $\QG$ is second countable (i.e.\ $C_0(\QG)$ is
        separable)  and let $\mathcal{I}=\{I_{\omega}: \omega \in \PUG\}$. Then each element in $\mathcal{I}$ is contained in a maximal element of $\mathcal{I}$.
\end{prop}

\begin{proof}
	The proof proceeds in two steps. First one shows a counterpart of Lemma 1.1 in \cite{Willis}, namely the following result: suppose $X\subset L^1(\QG)$ is a closed subspace and $\mathcal{F}\subset \PUG$ is a closed convex semigroup such that $I_{\omega}\subset X$ for $\omega \in \mathcal{F}$ and that for each finite subset $A\subset X$ and $\epsilon >0$ one can find $\omega_{A, \epsilon} \in \mathcal{F}$ such that for each $f \in A$ we have $d(f, I_{\omega_{A, \epsilon}})< \epsilon$. Then $X= I_\omega$ for some $\omega \in \mathcal{F}$. 
	The proof follows line by line as in \cite[Lemma 1.1]{Willis}. Secondly, one uses this fact to prove that every chain in $\mathcal{I}$ has an upper bound and concludes by the Kuratowski-Zorn lemma.
\end{proof}

\begin{theorem}
	Assume that $L^1(\QG)$ is separable. Consider the following
        list of conditions: 
	\begin{rlist}
		\item $\G$ is coamenable;
		\item $\G$ is amenable;
		\item for each $\omega \in \PUG$, the right ideal
                  $I_{\omega}$  admits a bounded left approximate  identity; 
		\item the collection $\mathcal{I}:= \{I_{\omega}: \omega \in \PUG\}$ admits a unique maximal element $I_{\textup{max}}$.
	\end{rlist}
	Then the following implications/equivalences hold:
        (ii)$\Longleftrightarrow$(iv), (i)$\Longrightarrow$(iii) and
        (i)+(ii)$\Longleftrightarrow$(iii)+(iv). 
	Moreover if (iv) holds then $I_{\textup{max}}=L^1_0(\QG)$; in
        particular (iv) holds if and only if the augmentation ideal
        belongs to $\mathcal{I}$. 
\end{theorem}

\begin{proof}
	We begin by proving the last statement. Suppose then that (iv) holds and we have the largest ideal $I_{\textup{max}} \in \mathcal{I}$. By Lemma \ref{preann} we have $I_{\textup{max}}^\perp =  \{x\in L^{\infty}(\QG): \forall_{\omega \in \PUG}\, L_{\omega}(x) =x \}$. Suppose then that $x \in L^\infty(\QG)$ and $L_\omega(x) = \omega(1)x$ for all $\omega \in L^1(\QG)$. This is equivalent to $\Com(x)= 1 \ot x$. As is well-known this holds if and only if $x \in \bc 1$ (see for example Result 5.13 of \cite{KV}) and $I_{\textup{max}} = L^1_0(\QG)$.

	The implication (ii)$\Longleftarrow$(iv) follows now from the forward implication of Theorem 4.2 of \cite{knr}. Further (iv)$\Longrightarrow$(ii)
	follows as in the backward implication of Theorem 4.2 of \cite{knr}, although the latter was formulated only for $\omega \in \PG$. Again, for completeness we sketch the proof. Suppose that (iv) holds, i.e.\ there is $\omega \in \PUG$ such that $I_\omega = L^1_0(\QG)$; equivalently, $\Fix L_\omega = \bc 1$. Pick a state $\mu \in L^1(\QG)$ and consider normal states $\omega_n:= S_n(\omega) \star \mu$, $n \in \bn$. Let $\gamma \in S(L^\infty(\QG))$ be the weak$^*$ limit point of the sequence $(\mu_n)_{n \in \bn}$ along some free ultrafilter. Then for every $x \in L^\infty(\QG)$ we have $R_\gamma(x) \in \Fix L_\omega = \bc 1$ so that we can define a state $\gamma'$ on $L^\infty(\QG)$ by the formula
	$\gamma'(x)1 = R_\gamma(x)$, $x \in L^\infty(\QG)$. One can then check that $\gamma = \gamma'$ and in fact $\gamma$ is a left invariant mean  on   $L^\infty(\QG)$.

	The implication (i)$\Longrightarrow$(iii) follows as in the intro to \cite{Willis}. Indeed, fix $\omega \in \PUG$. The fact that $\QG$ is coamenable implies that $L^1(\QG)$ admits a bounded approximate identity, say $(u_\lambda)_{\lambda \in \Lambda}$. Then one can check that the double-indexed net $(u_{\lambda} - S_n (\omega) \star u_\lambda)_{\lambda \in \Lambda, n \in \bn}$ is a bounded left approximate identity in $I_\omega$. 
	
	The implication (iii)+(iv) $\Longrightarrow$ (i)+(ii) follows from Proposition 16 of \cite{HNRSIN}; formally the formulation in \cite{HNRSIN} requires the two-sided bounded approximate identity in $L^1_0(\QG)$, but as $L^1_0(\QG)$ is an involutive Banach algebra (with the involution given by composing with the unitary antipode), the existence of a one-sided  bounded approximate identity implies that the two-sided one also exists.
\end{proof}


\section{Fixed points in $C_0(\G)$}


In this section, we consider the case when there are fixed points in $C_0(\G)$.
This only happens in the following cases: $\G$ is compact or $\omega$
is degenerate.

\begin{lem} \label{MGlimit}
Let $\omega \in \M^u(\QG)_1$ and suppose that
the  Ces\`aro averages $S_n(\omega)$ do not converge to $0$ weak$^*$ on
$C_0^u(\QG)$. 
Then the weak* limit 
\[
\wt\omega := \wlim_{n\to \infty} S_n(\omega)
\]
exists and $\wt\omega$ is a non-zero contractive idempotent in $M^u(\G)$.
\end{lem}

\begin{proof}
Since the  Ces\`aro sums $S_n(\omega)$ do not converge to $0$ weak$^*$ on
$C_0^u(\QG)$, there is some free ultrafilter $p$ such that  the limit
$\wt{\omega}_{p}$ is non-zero on $C_0^u(\G)$.
By Lemma~\ref{idempalternative}, $\wt{\omega}_{p}$ is a contractive
idempotent in $M^u(\G)$.
We shall show that in fact $\wt{\omega}_{p}$ is the weak* limit of
the sequence $(S_n(\omega))_{n=1}^\infty$ by showing that
every subnet of $(S_n(\omega))_{n=1}^\infty$ has a subnet converging
to $\wt{\omega}_{p}$ (see \cite[Exercise~11D.(c)]{willard}).
Since $(S_n(\omega))_{n=1}^\infty$ is bounded, every subnet of
$(S_n(\omega))_{n=1}^\infty$ has a subnet converging weak* 
to some $\wt{\omega}_{q}$, where $q$ is some free ultrafilter.
Let $a\in C_0^u(\G)$ be arbitrary. 
Since $\wt\omega_p \conv \omega = \wt\omega_p$, it follows that
$\omega^{\conv k}(L_{\wt\omega_p}^u(a)) = \wt\omega_p(a)$. 
Therefore
$S_n(\omega)\bigl(L_{\wt\omega_p}^u(a)\bigr) = \wt\omega_p(a)$
and so $\wt\omega_q(L_{\wt\omega_p}^u(a)) = \wt\omega_p(a)$.
It follows from Lemma~\ref{idempalternative}
that $\wt\omega_q$ is also a non-zero contractive idempotent in $M^u(\G)$.

Since $\wt\omega_q\in M^u(\G)$, we have by
Proposition~\ref{prop:collection} that
\[
\wt\omega_p = \wt\omega_q\circ L_{\wt\omega_p}^u = \wt\omega_p\conv \wt\omega_q
= \wt\omega_p\circ R_{\wt\omega_q}^u. 
\]
But
\[
\wt\omega_p\circ R_{\wt\omega_q}^u(a) =
\plim S_n(\omega)\bigl(R_{\wt\omega_q}^u(a)\bigr)  
= \wt\omega_q(a)
\]
as $\omega\conv \wt\omega_q = \wt\omega_q$. Consequently, $\wt\omega_q
= \wt\omega_p$, as required.
\end{proof}

\begin{remark}
The final condition in the following result may be interpreted
as a weak analogue to `the subgroup generated by the support
of $\omega$ is compact'.
Note that classically the closed sub\emph{semi}group generated by some
set is compact if and only if the closed subgroup generated by the
set is compact (because a compact semigroup with cancellation laws
is a group). See also Corollary~\ref{cor:non-deg}.
\end{remark}

The first result covers the general, possibly degenerate case.

\begin{prop}\label{prop:more-equiv}
Let $\omega \in \M^u(\QG)_1$. Then the following are equivalent:
\begin{rlist}
\item Ces\`aro sums $S_n(\omega)$ do not converge to $0$ weak$^*$ on
  $C_0^u(\QG)$;
\item for some (equivalently for every) free ultrafilter $p$ the functional $\wt \omega_p:= \plim S_n(\omega) $ is a non-zero contractive idempotent in $M^u(\QG)$;
\item $L_\omega^u$ has a non-zero fixed point in $C_0^u(\QG)$;
\item $L_{\omega}$ has a non-zero fixed point in $C_0(\QG)$;
\item there is a non-zero $\tau\in M^u(\QG)$ such that $\tau\conv
  \omega = \tau$; 
\item $L_{\omega}$ has a non-zero fixed point in $\linf(\QG)$ and
      there exists $e$ in $C_0^u(\QG)_+$ such that
      $|\omega|^{\conv k}.e = e.|\omega|^{\conv k} = |\omega|^{\conv k}$
      and  $ae - a\in N_{|\omega|^{\conv k}}$
      for every $k=1$, $2$, \ldots\ and $a\in C_0^u(\QG)$.
\end{rlist}
\end{prop}

The proof of this result will be  based on the following lemma.
For $\nu\in M(\G)$, we define
\[
R_\nu : C_0(\G)\to C_0^u(\G), \qquad
R_\nu(a) = (\id \ot \nu)\Cop(a).
\]
This is possible because \eqref{eq:reducing} implies that
$\Cop(a) \in M(C_0^u(\G)\ot C_0(\G))$ for every $a\in C_0(\G)$.

\begin{lem} \label{lemma:v-group-like-two}
Let $\omega\in \M^u(\G)_1$.
If $\lt_\omega$ has a non-zero fixed point in $C_0(\QG)$,
then there exists $v\in C_0(\G)$ such that
\[
\COP^{(k)}(v) - v^{\ot k+1} \in N_{|\omega|^{\ot k+1}} \qquad \text{and} \qquad
\omega^{\conv k} = v.|\omega|^{\conv k}
\]
for every $k \in \bn$.
Moreover, $|\omega|^{\conv k}(av^* v) = |\omega|^{\conv k}(v^* v a) = |\omega|^{\conv k}(a)$
for every $a\in C_0^u(\G)$ and $k \in \bn$.
\end{lem}

\begin{proof}
Suppose that $\lt_\omega$ has a non-zero fixed point $x$ in
$C_0(\G)$. Then we may choose $\nu\in\M(\G)_1$ such that
$\nu(x) = \|x\|$. Then $v = R_\nu(x^*)/\conj{\nu(x)}\in C_0^u(\G)$
satisfies $\|v\|\le 1$ and $\omega(v^*) = 1$.
Repeating the argument of Lemma~\ref{lemma:v-group-like}, we see that
this $v$ satisfies the statement (we may put in the proof there $\phi_n = \nu$
so that $v_n = v$ for all $n\in \bn$).
\end{proof}

\begin{proof}[Proof of Proposition~\ref{prop:more-equiv}]
  (i)$\implies$(ii)(for every free ultrafilter)  follows from Lemma~\ref{MGlimit}.
  
  (ii)(for a fixed free ultrafilter) $\implies$(iii) is clear because
  every element in $L_{\tilde\omega_p}^u(C_0(\QG))$
  is a fixed point of  $\Fix L_\omega^u$.
  (Note also that the map $\mu \mapsto L_{\mu}^u$ is injective.)

(iii)$\implies$(i):
Let $a\in C_0^u(\QG)$ be a non-zero fixed point of 
$L_{\omega}^u$. 
Consider $\nu \in \M^u(\G)$ such that $\nu(a) \neq 0$. Then we have
\[
\omega^{\conv n}(R_{\nu}(a)) = \nu (L_{\omega^{\conv n}} (a)) = \nu
({L_{\omega}}^n (a))  = \nu (a),
\]
and it follows that $S_n(R_\nu(a)) = \nu(a)\ne 0$ for every $n\in \bn$.

(ii)(for a fixed free ultrafilter)$\implies$(iv):
If $\tilde\omega_p$ is a non-zero contractive idempotent in $M^u(\QG)$,
then $L_{\tilde\omega_p}(C_0(\QG)) = \Fix L_\omega \cap C_0(\QG)$ is  
non-zero (see Lemma~\ref{lem:fix_L}).

(iv)$\implies$(i)
is similar to ``(iii)$\implies$(i)'' when we take $\nu\in M(\G)$
and  $R_\nu: C_0(\G)\to C_0^u(\G)$ is the map defined before
Lemma~\ref{lemma:v-group-like-two}.

(ii)(for a fixed free ultrafilter)$\implies$(v): Take $\tau = \wt\omega_p$.

(v)$\implies$(iii): If $a\in C_0^u(\G)$ is such that $L_\tau^u(a)\ne 0$,
then the latter is a fixed point of $L_\omega^u$ as
$L_\omega^u\circ L_\tau^u = L_{\tau\conv \omega}^u = L_\tau^u$.

(iv)$\implies$(vi):
Suppose that $L_\omega$ has a non-zero fixed point $x\in C_0(\G)$.
Then we can apply Lemma~\ref{lemma:v-group-like-two}
to obtain a suitable $v\in C_0^u(\G)$.
As (iv) implies (i), we know that the limit $\wt\omega := \wlim_{n\to \infty} S_n(\omega)$ exists by Lemma \ref{MGlimit}. We shall then  show that $|\widetilde\omega|(a) = |\omega|\latetilde(a)$
for every $a\in C_0^u(\G)$. 
Note  that $\omega^{\conv k}(a) = |\omega|^{\conv k}(av)$ by
Lemma~\ref{lemma:v-group-like-two}, and
hence
\[ 
\widetilde\omega(a)
= \lim \frac1n \sum_{k=1}^n \omega^{\conv k}(a)
= \lim \frac1n \sum_{k=1}^n |\omega|^{\conv k}(av)
= |\omega|\latetilde(av).
\] 
Recalling that $|\omega|^{\conv k}(av^* v) = |\omega|^{\conv k}(a)$ for all $k \in \bn$, it follows
similarly that $\widetilde\omega(av^*) = |\omega|\latetilde(a)$.
These identities imply that
$\|\widetilde\omega\| = \|\,|\omega|\latetilde\|$.
Moreover, by the Cauchy--Schwarz inequality,
\[
|\widetilde\omega(a)|^2 \le |\omega|\latetilde(aa^*)|\omega|\latetilde(v^*v)
\le |\omega|\latetilde(aa^*).
\]
It now follows from the uniqueness of
absolute value (Proposition III.4.6 of \cite{takesaki:vol1})
that $|\widetilde\omega| =  |\omega|\latetilde$.

Since  $x\in C_0(\G)$ is a fixed point of $\lt_\omega$,
we have $\lt_{\widetilde\omega}(x) = x$.
Since  $\widetilde\omega$ is a contractive idempotent in $\M^u(\QG)$
it follows from Lemma~3.1 of \cite{NSSS} (see also  equation (2.5) of
\cite{kasprzak}) that 
\[
x x^* = \lt_{\widetilde\omega}(x)\lt_{\widetilde\omega}(x)^*
= \lt_{|\widetilde\omega|}(x\lt_{\widetilde\omega}(x)^*) =
\lt_{|\widetilde\omega|}(xx^*).
\]
Pick $\nu\in\M(\QG)_+$ such that $\nu(xx^*) = \|xx^*\|$.
Then $R_\nu(xx^*)\in C_0^u(\G)$ is a fixed point of $L_{|\wt{\omega}|}^u$
and $|\wt\omega|(R_\nu(xx^*)) = \nu(xx^*)$.
Put $e = R_\nu(xx^*)/\nu(xx^*)\in C_0^u(\G)$.
Since $|\wt\omega|$ is an idempotent state in $\M^u(\QG)$
and $e\in L_{|\wt\omega|}^u(C_0^u(\QG))$,
$e$ is in the multiplicative domain of $|\wt\omega|$
by Lemma~2.5 of \cite{salmi-skalski:idem}.

Since $|\widetilde\omega| = |\omega|\latetilde$,
we have for any $k\in \bn$ that
\[
\lt_{|\omega|^{\conv k}}(xx^*) = \lt_{|\omega|^{\conv k}}(\lt_{|\widetilde\omega|}(xx^*)) =
\lt_{|\widetilde\omega|\conv|\omega|^{\conv k}}(xx^*) = \lt_{|\widetilde\omega|}(xx^*) = xx^*
\]
and so $|\omega|^{\conv k}(e) = 1$. Hence
\[
1 = |\omega|^{\conv k}(e)^2 \le |\omega|^{\conv k}(e^2) \le 1
\]
(by Cauchy--Schwarz) and it follows that
$e$ is also in the multiplicative domain of $|\omega|^{\conv k}$ for every $k \in \bn$.
Now for every $a\in C_0^u(\G)$ and $k\in \bn$
\[
|\omega|^{\conv k}\bigl((ae - a)^*(ae - a)\bigr)
= |\omega|^{\conv k}(ea^*ae) - |\omega|^{\conv k}(a^*ae)
  -|\omega|^{\conv k}(ea^*a) + |\omega|^{\conv k}(a^*a) = 0
\]
because $e$ is in the multiplicative domain of $|\omega|^{\conv k}$.
Therefore $ae-a\in N_{|\omega|^{\conv k}}$.

(vi)$\implies$(i):
Fix a free ultrafilter $p$ and let $\wt\omega_p\in\ruc^u(\QG)^*$
be the weak* limit of $S_n(\omega)$ along $p$.
Since $L_\omega$ has a non-zero fixed point in $\linf(\G)$,
there is, by Lemma~\ref{lemma:cesaro-nilpotent}, 
$x\in\ruc^u(\QG)$ such that  $\wt\omega_p(x)\ne 0$.
Let $v\in C_0^u(\G)^{**}$ be as in Lemma~\ref{lemma:v-group-like}.
It follows from (vi) that
$|\omega|^{\conv k}.e = |\omega|^{\conv k}$ on $C_0^u(\G)^{**}$
for every $k\in \bn$. Therefore
\[
\wt\omega_p(ex) = \plim S_n(\omega)(ex)
= \plim S_n(|\omega|)(exv)
= \plim S_n(|\omega|)(xv)
= \plim S_n(\omega)(x)
= \wt\omega_p(x) \ne0.
\]
Since $ex\in C_0^u(\QG)$, (i) holds.
\end{proof}

The above theorem has an immediate corollary, which can be also proved directly.
\begin{cor}
If  $\omega \in \MUG_1$ and  $\omega^{\conv n}\to 0$ in the weak*
topology of $M^u(\G)$, then $\Fix L_{\omega} \cap C_0(\G) = \{0\}$. 
\end{cor}

In the non-degenerate case, Proposition~\ref{prop:more-equiv}
can be used to completely characterise the fixed points in $C_0(\G)$. Note that the second statement can be deduced from Theorem \ref{thm:group-like} and the analogous fact for $\omega$ being positive. 

\begin{cor} \label{cor:non-deg}
Suppose that $\omega\in \M^u(\QG)_1$ is
non-degenerate. If $\G$ is not compact,
then $\Fix L_\omega\cap C_0(\G) = \{0\}$.  
If $\G$ is compact, then $\Fix L_\omega$ is either $\{0\}$ or
$\C u$ where $u$ is a group-like unitary in $C(\G)$.
\end{cor}

\begin{proof}
For non-degenerate $\omega$, the element $e$ in statement (iv)
of Proposition~\ref{prop:more-equiv} is a positive right identity,
hence an identity. Therefore $\G$
is necessarily compact if $L_\omega$ has a non-zero fixed point in
$C_0(\G)$.

Now suppose that $\G$ is compact.
By Theorem~\ref{thm:group-like} and Lemma~\ref{lemma:cesaro-nilpotent},
if $\Fix L_\omega\ne\{0\}$ there is a group-like unitary  $u\in C(\G)$
such that $\Fix L_\omega = \Fix L_{|\omega|} u$. By \cite[Theorem 3.6]{knr},
$\Fix L_{|\omega|} = \C 1$.
\end{proof}

Finally we note that for discrete quantum groups and positive functionals the above corollary can be strengthened: in \cite[Theorem 2]{Mehrdad1} it is shown that if $\QG$ is discrete and infinite (in other words non-compact) and $\omega\in \PUG$ is non-degenerate then for every $x\in C_0(\QG)$ we have $L_{\omega^{\star n}}(x) \stackrel{n \to \infty}{\longrightarrow}0$.


\section{Fixed points in $L_p(\G)$ for tracial Haar weights}


Yau showed in \cite{yau} that any harmonic function $f \in L_p(M)$ on
a complete manifold $M$ is constant, for any $p \in ( 1 , \infty)$.
Motivated by this result, Chu \cite{chu} introduced and studied
the  space of $L_p$-fixed points of the convolution operator
$L_\omega$ for $p \in [ 1 , \infty )$. The main result of \cite{chu}
  states that if $\omega$ is an adapted probability measure, then any
  such fixed point must be a constant
  function,  see \cite[Theorem 3.12, Corollary 3.14]{chu}. 

A quantum group version of Chu's result has been obtained by Kalantar in \cite{kal-lp}. More precisely, let $\mathbb{G}$ be a locally compact quantum group 
with tracial (left) Haar weight, $p \in [ 1 , \infty )$, and $\omega \in P(\mathbb{G})$ a non-degenerate quantum probability measure. 
Consider the space of $\omega$-harmonic vectors in the non-commutative $L_p$-space $L_p(\G)$: 
$$H^p_\omega (\G) := \{ f \in L_p(\G) \mid L_\omega f = f \} .$$
If $\mathbb{G}$ is non-compact, then $H^p_\omega (\G) = \{0\}$ 
\cite[Theorems 2.4, 2.6]{kal-lp}. If $\mathbb{G}$ is compact, then $H^p_\omega (\G) = \mathbb{C} 1$ \cite[Theorem 2.8]{kal-lp}.

We will now generalise these results to our setting, extending the
context from $M(\mathbb{G})$ to $\MUG$ and allowing non-positive
quantum measures. We begin by outlining the construction of the action
of the convolution operators on the non-commutative $L_p$-spaces, with
$p \in [1,\infty)$.  Recall that we denote by $\varphi$ the tracial
  left Haar weight, and by $\psi = \varphi R$ the right Haar weight,
  where $R$ is the unitary antipode of $L^\infty(\G)$. We write  
$L_p(\G)$ and $\tilde{L}_p(\G)$ for the non-commutative $L_p$-spaces
  associated with $\varphi$ and $\psi$, respectively. Recall that
  these are defined as the completions of  
$\mathcal{M}_\varphi$ and $\mathcal{M}_\psi$ under the norms $\|x\|_p
  = \varphi (|x|^p)^{\frac{1}{p}}$ and
  $\|x\|_p = \psi (|x|^p)^{\frac{1}{p}}$, respectively; here,  
$$\mathcal{M}_\varphi := \mathrm{lin} \{ x \in L^\infty(\G)^+ \mid
  \varphi (x) < \infty \} .$$  
For $p,q \in ( 1 , \infty )$ with $\frac{1}{p} + \frac{1}{q} = 1$, the
spaces $L_p(\G)$ and $\tilde{L}_q(\G)$ are  
the duals of each other, via 
$$\langle a , b \rangle = \varphi (a R(b)) = \psi (R(a)b) \quad (a \in
L_p(\G), b \in \tilde{L}_q(\G)) ;$$  
cf.\ \cite[p.\ 3972]{kal-lp}. Recall that we denote by $L^1(\G)$ the
predual of $L^\infty(\G)$; the spaces $L^1(\G)$ and $L_1(\G)$ are
isometrically isomorphic via the map $\Phi: L_1(\G) \to L^1(\G)$ such
that for $x \in  \mathcal{M}_\varphi$ and $y \in L^\infty(\G)$ we have
$\Phi(x) (y)= \varphi(x y)$.

The convolution action of $M(\mathbb{G})$ on $L_p(\mathbb{G})$ is
defined by complex interpolation; for $\omega \in P(\mathbb{G})$ this
is discussed in  
in \cite[section 2]{kal-lp}, cf.\ also \cite[p.\ 19, in particular
  Lemma 4.3]{br-ru}. Below we will outline what is needed to extend
this to $\omega \in \MUG$. 

Begin by recalling the argument appearing in \cite{kal-lp}: if $\omega
\in M(\QG)_+$, then the convolution operator $L_\omega:L^\infty(\G)\to
L^\infty(\G)$ preserves the weight $\phi$ (so also the space
$\mathcal{M}_\varphi$).  Further Kalantar checks in \cite{kal-lp} that
we have the following equality: 
\begin{equation}
  \Phi \circ L_\omega(z) = (\omega \circ R)\star \Phi(z)
  \qquad (z \in \mathcal{M}_{\varphi}).
\label{unitantipode}
\end{equation} 
This implies that in fact $L_\omega$ yields (by continuous extension)
a bounded map on $L_1(\QG)$, and hence by complex interpolation on all the spaces
$L_p(\G)$. Further the above formula is true (simply by linearity and
continuity) for all $\omega \in M(\G)$ and $z \in L_1(\G)$. A similar
procedure shows that the right convolution operators $R_\omega$ act
boundedly on $\tilde{L}_q(\G)$. 

Further Kalantar shows in the proof of  \cite[Theorem 2.2]{kal-lp}
that the following is true for all $p\in (1, \infty), \omega \in
P(\G)$, $f \in L_p(\G)$ and $g \in \tilde{L}_{q}(\G)$:  
\begin{equation} \label{adjoint}
  \langle L_\omega (f), g \rangle = \langle f, R_\omega (g)\rangle.
\end{equation}
Before we extend these formulas to operators associated with $\omega
\in \MUG$, we need to recall that for every $p \in (1, \infty)$ we
have that each of the subspaces $L_1(\QG) \cap L_p(\QG)$ and
$L^\infty(\QG) \cap L_p(\QG)$ is dense in $L_p(\QG)$ (with similar
statements holding for $\tilde{L}_p(\QG)$). Finally note one more easy
property:   the span of $\{L_\nu(z): \nu
  \in L^1(\G), z\in \mathcal{M}_\varphi\}$ is dense in $L_1(\G)$ (and
  naturally the span of $\{R_\nu(z): \nu \in L^1(\G), z\in
  \mathcal{M}_\psi\}$ is dense in $\widetilde{L}_1(\G)$).  Indeed,
  it suffices to use the formula \eqref{unitantipode},
  the fact that $R$ is an isometry on $L^\infty(\G)$ and finally the
  fact that the linear span of $L^1(\G) \star L^1(\G)$ is dense in
  $L^1(\G)$.

\begin{lem}
Let $\omega \in \MUG$, $p \in [1,\infty)$. Then the operator
  $L_\omega:L^\infty(\G) \to L^\infty(\G)$ defines by
  restriction/continuous extension a bounded operator on $L_p(\G)$ (to
  be denoted by the same symbol). Further we have the following
  equalities: 
  \begin{equation} \label{form1}
    \Phi \circ L_\omega(z) = (\omega \circ  R^u)\star \Phi(z), \qquad
    z \in L_1(\G), 
\end{equation} 
\begin{equation}  \langle L_\omega (f), g \rangle = \langle f,
  R_\omega (g)\rangle, \qquad f \in L_p(\G), g \in
  \tilde{L}_{q}(\G). \label{form2} 
\end{equation} 
\end{lem}
\begin{proof}
	We follow the line of argument in \cite{kal-lp}. Note we can
        (and do) assume that $\omega \in \PUG$ and then argue by
        linearity.  
The fact that $L_\omega$ preserves the left Haar weight is Lemma 3.4
of \cite{knr}. In the next step we show that the formula
\eqref{unitantipode} holds for $z$ in a dense subset of
$L_1(\QG)$. Indeed, take $z=L_\nu (z')$, where
$z' \in \mathcal{M}_\varphi$ and $\nu \in \lone(\G)$. Then we have 
\begin{align*}
  \Phi \circ L_\omega(z) &= \Phi \circ L_\omega(L_\nu (z'))
  = \Phi \circ L_{\nu \star\omega} (z')
  = ((\nu \star \omega)\circ R) \star \Phi(z')
  = ((\omega \circ R_u)\star (\nu \circ R))\star \Phi(z') \\
  &= (\omega \circ R_u)\star ((\nu \circ R)\star \Phi(z'))
  = (\omega \circ R_u)\star \Phi \circ L_{\nu}(z')
  = (\omega \circ R_u)\star \Phi(z)
\end{align*}
where in the third and the sixth equalities we used
\eqref{unitantipode} and in the fourth one \eqref{convantipode}. 
Now, as the span of $L_\nu (z')$ with $z' \in \mathcal{M}_\varphi$ and
$\nu \in \lone(\G)$ is dense in $L_1(\QG)$, we deduce that $L_\omega$
restricts/extends to a bounded map on the whole $L_1(\QG)$ and further
\eqref{form1} holds for all $x \in L_1(\QG)$.  
Now that we have established that $L_\omega$ is bounded on $L_1(\QG)$
(as well as on $\linf(\QG)$), it follows by complex interpolation that
$L_\omega$ is bounded also on $L_p(\QG)$ for $p\in(1, \infty)$.
Similarly, $R_\omega$ is bounded on all $\tilde L _q(\QG)$.

Consider now the second statement and assume first that $p=1$. Suppose
that $f=L_\nu(f')$ for some $\nu \in L^1(\QG)$ and $f'\in L_{1}(\G)$
and $g \in \tilde{L}^{\infty}(\G)$. We compute: 
\begin{align*}
  \langle L_\omega (f), g \rangle
  &= \langle L_\omega (L_\nu(f')), g \rangle
  = \langle L_{\nu \star\omega} (f'), g \rangle
  = \langle f', R_{\nu \star\omega}(g) \rangle\\
  &= \langle f', R_{\nu} (R_\omega(g)) \rangle
  = \langle L_{\nu}(f'),  R_\omega(g) \rangle
  = \langle f, R_\omega(g) \rangle, 
\end{align*} 
where in the third and fifth equalities we used \eqref{adjoint}. By
continuity, \eqref{form2} follows (for $p=1$).

Let then $p \in (1, \infty)$. As each duality is given by computing
the trace on the respective (strong) products of measurable operators,
we have by the above  
\[
\langle L_\omega (f), g \rangle = \langle f,
R_\omega (g)\rangle
\] 
for all  $f \in L_p(\G) \cap L_1(\QG)$,
$g \in \tilde{L}_{q}(\G) \cap L^\infty(\QG)$. Then the general
statement follows by first approximating a general $f \in L_p(\QG)$ by
$f_i \in L_p(\G) \cap L_1(\QG)$ and then a general
$g \in \tilde{L}_q(\QG)$ by $g_i \in \tilde{L}_{q}(\G)\cap L^\infty(\QG)$.
This uses the continuity of the  operators $L_\omega$ and $R_\omega$
on respectively $L_p(\QG)$ and $\tilde{L}_q(\QG)$.
\end{proof}

We are ready for the main result of this section.

\begin{tw} \label{prop:Lpharmonic}
Let $\mathbb{G}$ be a locally compact quantum group with tracial
(left) Haar weight, let $p \in [ 1 , \infty )$, and let $\omega \in
  M(\mathbb{G})_1$ be non-degenerate.  
\begin{itemize} 
		\item[(i)] Assume $\G$ is non-compact. Then
                  $H^p_\omega (\G) = \{0\}$.  
		\item[(ii)] Assume $\G$ is compact and co-amenable. If
                  $L_\omega$ has a non-zero fixed point in $L_p(\G)$,
                  then $H^p_\omega (\G) = \mathbb{C} u$ where $u$ is a
                  group-like unitary in $C(\G)$.  
	\end{itemize} 
\end{tw} 
\begin{proof} 
	(i) Case 1: $p > 1$. Let $f \in H^p_\omega (\G)$, and $g \in
  \tilde{L}_q(\G)$, $g= R_\nu(g')$ for some $\nu \in L^1(\QG)$ and
  $g' \in \tilde{L}_{q}(\G)$. Then the element $\Omega_{f,g'} \in
  L_\infty(\G)$ defined through
  \[
  \langle \Omega_{f,g} , h \rangle = \langle L_h(f) , g' \rangle \quad
  (h \in L^1(\G))
  \] 
	belongs to $C_0(\G)$; cf.\ \cite[proof of Theorem 2.4]{kal-lp} (note that the latter argument assumes $p \in (1,2]$, but it can be modified to work for all $p>1$ by using the density result obtained in 
	\cite[Theorem 8]{masuda} to approximate $f$ and $g$). Note also that $S_n(\omega) \to 0$ ($w^*$) since otherwise, by 
	our Proposition \ref{prop:more-equiv}, $L_\omega$ would have a non-zero fixed point in $C_0(\G)$, 
	which would contradict the non-compactness of $\G$, in view of our Corollary \ref{cor:non-deg}. Thus we obtain, using twice \eqref{form2}: 
	\begin{align*}\langle f , g \rangle &= \langle L_{S_n(\omega)} f , g \rangle = \langle f , R_{S_n(\omega)} (g) \rangle =\langle f , R_{S_n(\omega)} (R_\nu(g')) \rangle = \langle f , R_{S_n(\omega) \star \nu} (g') \rangle = 
	\langle L_{S_n(\omega) \star \nu} (f) ,  g' \rangle \\&= (S_n(\omega) \star \nu) (\Omega_{f,g'}) = S_n(\omega) ((\id \ot \nu)\cop(\Omega_{f,g'}))
	 \stackrel{n \to \infty}{\longrightarrow} 0\end{align*}
	whence $f=0$ (since $g$ was an arbitrary element of a dense subset of $\tilde{L}_q(\G)$).

Case 2: $p=1$. Let $f \in H^1_\omega (\G)$. Note that, as in 
Case 1, we have that $S_n(\omega) \to 0$ ($w^*$). By
\eqref{form1} we have for any $n \in \bn$ the following
equality: 
\[
\Phi (f) = \Phi \circ L_{S_n(\omega)}(f)
= (S_n(\omega) \circ R_u) \star \Phi(f)
\]	
Thus for all $a \in C_0(\G)$: 
\begin{align*}
  \Phi(f)(a) &= \bigl((S_n(\omega) \circ R_u) \star \Phi(f)\bigr)(a)
  = (S_n(\omega) \circ R_u) ((\id \ot \Phi(f))\cop(a)) \\
  &= S_n(\omega) \left( R_u ((\id \ot \Phi(f))\cop(a))\right)
  \stackrel{n \to \infty}{\longrightarrow} 0
\end{align*}
so that $\Phi(f) = 0$ and of course $f=0$.

(ii) Since $\G$ is compact, we have
$L^\infty(\G) \subseteq L_p(\G) \subseteq L_1(\G)$.
Let $f \in H^p_\omega (\G) \setminus \{0\}$.
As $\G$ is co-amenable, we have a bounded approximate identity
$(e_\lambda)_{\lambda \in \Lambda}$ in $L^1(\G)$. For all $n \in
\mathbb{N}$, choose $e_{\lambda_n}$ such that
$\|  \Phi(f) \star  e_{\lambda_n}  - \Phi(f) \|_1 < \frac{1}{n}$.
Since  	$L^\infty(\G)$ is dense in $L_1(\G)$, we can further find $x_n \in L^\infty(\G)$ with $\|\Phi( x_n) - e_{\lambda_n} \|_1 < \frac{1}{n \| f \|_1}$. Hence, for all $n \in \mathbb{N}$ we obtain: 
	$$\|  \Phi(f) \star \Phi(x_n)  - \Phi(f) \|_1 \leq \| \Phi(f) \star \Phi(x_n) -  \Phi(f) \star  e_{\lambda_n}  \|_1 + \|  \Phi(f) \star  e_{\lambda_n}  - \Phi(f) \|_1 < \frac{2}{n} ,$$ 
	so $f_n := \Phi^{-1} (\Phi(f) \star \Phi(x_n)) \stackrel{n \to \infty}{\longrightarrow} f$ in $L_1(\G)$. 
	We can now exploit the formula \eqref{unitantipode} to see that 
	\[ \Phi(f)\star \Phi(x_n) = ((\Phi(f)\circ R)\circ R)\star \Phi(x_n) = \Phi (L_{\Phi(f)\circ R} (x_n)).\]
Thus $f_n=L_{\Phi(f)\circ R} (x_n) \in  L^1(\G) \star L^\infty(\G)  = RUC(\G) = C(\G)$. Finally 
\[ L_\omega \circ L_{\Phi(f)\circ R} (x_n) = L_{(\Phi(f)\circ R) \star \omega}(x_n)\]
and using the \eqref{unitantipode}  once again we check that 
\[ (\Phi(f)\circ R) \star \omega = ((\omega \circ R) \star \Phi(f)) \circ R = \Phi (L_\omega(f)) \circ R =  
	\Phi (f) \circ R.\]	
As $f_n \to f \not= 0$, we have $f_n \not= 0$ for $n$ large enough,
	 so that $L_\omega$ has a non-zero fixed point 
	in $C(\G)$. Thus, our Corollary \ref{cor:non-deg} implies, since $\G$ is compact, that $\mathrm{Fix}\, L_\omega = \mathbb{C} u$, where $u$ is a group-like unitary in $C(\G)$. 
	So, $f_n \in \mathbb{C} u$ for $n$ large enough, whence $f \in \mathbb{C} u$. Thus, we have $H^p_\omega (\G) \subseteq \mathbb{C} u$. The reverse inclusion is clear as 
	$u \in \mathrm{Fix}\, L_\omega \subseteq L^\infty(\G) \subseteq L^p(\G)$. 
\end{proof}


\section{Classical case}


In this section we discuss the classical case, i.e.\ the situation where $G$ is a locally compact group.

Given a probability measure 
$\omega\in \Prob(G)$, we let $S_\omega$ and $G_\omega$ denote,
respectively, the closed semigroup 
and group generated by $\supp\omega$.  We say that
$\omega$ is {\it non-degenerate} (or {\it irreducible}) if $S_\omega=G$, and
{\it adapted} if $G_\omega=G$; so that non-degeneracy implies
adaptedness, but the converse implication does not hold. Given a general  $\omega \in \M(G)$ we will call it non-degenerate (respectively, adapted) if $|\omega|$ is non-degenerate (respectively, adapted).  It is easy to
see that this notion of non-degeneracy coincides with the one
introduced earlier. 
In the non-degenerate case, Theorem~\ref{thm:group-like} says
that if $\lt_\omega$ has a non-zero fixed point in $\luc(G)$,
then there is $\chi\in \dual G$ such that
$\omega = \chi|\omega|$ and
$\Fix \lt_\omega = (\Fix \lt_{|\omega|}) \conj{\chi}$.
The following theorem addresses this result in the degenerate situation.

\begin{theorem} \label{thm:fixed-in-luc}
	Suppose that $G$ is a locally compact group and that
	$\omega\in\M(G)_1$.
	If $\lt_\omega$ has a non-zero fixed point in $\luc(G)$,
	then there is a continuous character $\chi:S_{|\omega|}\to\bt$ such that
	$\omega = \chi |\omega|$.
\end{theorem}

\begin{proof}
	Let  $f\in\luc(G)$ be a non-zero fixed point of
	$\lt_\omega$.  Let $\widetilde f$ denote the extension of $f$ to the
	LUC-compactification $G^\luc$. Then there exists $x\in G^\luc$ such that 
	$\|f\| =|\widetilde{f}(x)|$.  Multiplying $f$ by a scalar we may and shall assume that
	$1=\widetilde{f}(x)=\|f\|$.  
	
	Write $\omega = u |\omega|$ where
	$u$ is a unimodular, measurable function.
	With $x$ as above, which can be viewed also as an element of $\luc(G)^*$, we define a function $x\odot f $ by 
	\[
	(x\odot f) (s) = \la x, \ell_s f \ra = \widetilde{f}(sx) \qquad(s\in G)
	\]
	where $sx$ is defined using the multiplication on $G^\luc$. 	Then $x\odot f\in \luc(G)$ and
	for any $n \in\bn$, if we choose a net $(t_i)_{i \in \Ind}$ of elements of $G$ convergent to $x$ inside $G^\luc$, we have
	\[
	\la \omega^{\conv n}, x\odot f \ra = \la \omega^{\conv n}\cdot x, f \ra
	= \lim_{i \in \Ind} \la \omega^{\conv n} \cdot t_i, f \ra =  \lim_{i \in \Ind} \la \omega^{\conv n}, r_{t_i} f \ra 
	= \lim_{i\in \Ind} (\lt_\omega^n f)(t_i)
	= \la x, \lt_\omega^n f \ra = \la x, f\ra,
	\]
	where we first used the definition of the product in $\luc(G)^*$ and then the fact that $\omega^{\star n}$, as a measure, belongs to the topological centre of $\luc(G)^*$ (see for example \cite{Wong}).
	Hence letting $S=\supp|\omega|$ we have that
	\[
	\widetilde{f}(x) = \int_{S^{\times n}} \widetilde{f}(s_1\dots s_nx) u(s_1)\dots u(s_n) \, d|\omega|^{\times n}(s_1,\dots,s_n).
	\]
	Since $|\widetilde f|$ attains its maximum at $x$, and $\widetilde{f}(x)=1$, 
	\[
	\widetilde{f}(s_1\dots s_nx) u(s_1)\dots u(s_n) = \widetilde{f}(x) =1
	\]
	for $|\omega|^{\times n}$ almost every $(s_1,\dots,s_n)$.  Now
	\[
	\conj{\chi(s)} =\widetilde{f}(sx)
	\]
	defines a bounded continuous function $\chi$ on $G$.
	Taking $n=1$ above we see that
	\[
	\chi(s) = u(s)\qquad  \text{for }|\omega|\text{-almost every }s,\text{ and hence }\chi(S)\subseteq\bt.
	\]
	For arbitrary $n$ we then see that
	\[
	\chi(s_1\dots s_n)=u(s_1)\dots u(s_n)\quad \text{for }|\omega|^{\times n}\text{-almost every }(s_1,\dots,s_n)
	\]
	while the $n=1$ case implies that
	$u(s_1)\dots u(s_n)=\chi(s_1)\dots \chi(s_n)$ for $|\omega|^{\times n}$-almost every $(s_1,\dots,s_n)$.
	Since $\chi$ is continuous we conclude that
	\[
\chi(s_1\dots s_n)=\chi(s_1)\dots \chi(s_n)\quad\text{for every } n \in \bn \text{ and } (s_1,\dots,s_n)\in S^{\times n}.
	\]
	We let $S^n\subseteq G$ denote the set of products from $S^{\times n}$.  An obvious
	induction shows that $\chi$ is multiplicative on the semigroup $S_{|\omega|}'=\bigcup_{n=1}^\infty S^n$, and
hence, by continuity, $\chi$ is multiplicative on its closure $S_{|\omega|}$.
\end{proof}

Notice  that if $|\omega|$ in the last theorem is non-degenerate then
the character $\chi$ with the desired properties is defined on the
whole of $G$.  In such a situation, if $f\in\Fix \lt_{|\omega|}$, then
it is clear that $f\conj{\chi}\in \Fix \lt_\omega$. 
In other cases there is no evident analogue of this fact.  
For any closed subsemigroup $S$ of $G$ and continuous character
$\chi:S\to\bt$ we let 
\begin{equation} \label{def:CDfp}
\luc_{\chi,S}(G)=\{f\in\luc(G):f(st)=\conj{\chi(s)}f(t)\text{ for }s\in S,t\in G\}
\end{equation}
If $\omega=\chi|\omega|$ as in the theorem above, then  it is easy to see that
\[
\luc_{\chi,S_{|\omega|}}(G)\subseteq \Fix \lt_\omega.
\]
If $\omega=|\omega|$ is a probability measure, then in many cases -- e.g.\ $G$ abelian (\cite{ChD})
or $G$ is SIN and $\omega$ is non-degenerate (\cite{Jaworski}) --
we have the Choquet-Deny theorem:  $\Fix \lt_\omega=\luc_{1,S_{|\omega|}}(G)$.

Suppose we are given  a  closed subsemigroup $S$ densely generating a
locally compact group $H$. 
Except in the case of abelian groups, it is not evident how to extend
a character $\chi_0:S\to\bt$ to a character on $H$.  Even in the abelian case
it is not clear that the extension can be assumed continuous if
$\chi_0$ is continuous.   

\begin{lem} \label{lem:CD-ext-to-char}
Let $S$ be a closed subsemigroup of $G$, let $H$ be the smallest
closed subgroup of $G$ containing $S$, and let $\chi_0:S\to\bt$  be a
continuous character. If there exists a non-zero bounded continuous function $f$ such that 
$f(st)=\conj{\chi_0(s)}f(t)$  for all $s\in S, t\in G$ then
$\chi_0$ extends to a continuous character $\chi$ on $H$. Furthermore  $f(st)=\conj{\chi(s)}f(t)$  for all $s\in H, t\in G$.
\end{lem}

\begin{proof}
	Set $U_f=\{t\in G:f(t)\not=0\}$.
	Let $s\in S$ and $t\in U_f$.  We have that $\conj{\chi_0(s)}f=\ell_s f$ (left translation)
	and hence $\ell_{s^{-1}}f=\chi_0(s)f$ which implies that
	\[
	\conj{\chi_0(s)}=\frac{f(st)}{f(t)}\quad\text{and}\quad\chi_0(s)=\frac{f(s^{-1}t)}{f(t)}\quad\text{so}\quad
	\langle S\rangle U_f\subset U_f
	\]
	where $\langle S\rangle$ is the subgroup generated (algebraically) by $S$.
	We fix some $t_0\in U_f$ and define $\chi:G\to\bc$ by
	\[
	\conj{\chi(s)}=\frac{f(st_0)}{f(t_0)}
	\]
	so that $\chi$ is continuous.  Then for $s,s'$ in $S\cup S^{-1}$ we have that
	\[
	\conj{\chi(ss')}=\frac{f(ss't_0)}{f(s't_0)}\frac{f(s't_0)}{f(t_0)}=\conj{\chi(s)\chi(s')}.
	\]
	As simple induction then shows that $\chi$ is multiplicative on $\langle S\rangle=\bigcup_{n=1}^\infty
	(S\cup S^{-1})^n$, so also on  $H$.  Similarly, we see that $\ell_s f=\conj{\chi(s)} f$ on $\langle S\rangle$, and
	hence also on $H$, which ends the proof.
\end{proof}

We consider below two cases in which the Choquet--Deny theorem 
always holds in the context of probability measures, and show that it also holds in the more general framework of contractive measures we consider here. 
 For the abelian case we adapt the remarkably simple proof of \cite{raugi},
and for the weakly almost periodic case we adapt a proof from \cite{temp}, which Tempel'man attributes
to Ryll-Nardzewski. Below $WAP(G)$ denotes the space of weakly almost periodic functions on $G$; recall also the notation introduced in \eqref{def:CDfp}.

\begin{theorem} \label{thm:fixed-abelian-and-wap}
	Let $G$ be a locally compact group and $\omega\in\M(G)_1$.
	
	{(i)} If $G$ is abelian and
	$\lt_\omega$ has a non-zero fixed point in $\luc(G)$,
	then there is a continuous character $\chi:G_{|\omega|}\to\bt$ such that
	$\omega = \chi |\omega|$ and 
	\[
	\Fix \lt_\omega=\luc_{\chi,G_{|\omega|}}(G)
	\]
	
	{(ii)}  If $\lt_\omega$ has a non-zero fixed point in $\WAP(G)$ 
	then there is a continuous character $\chi:G_{|\omega|}\to\bt$ such that
	$\omega = \chi |\omega|$ and 
	\[
	\Fix \lt_\omega\cap \WAP(G)=\WAP_{\chi,G_{|\omega|}}(G)
	\]
	where $\WAP_{\chi,G_{|\omega|}}(G)=\WAP(G)\cap \luc_{\chi,G_{|\omega|}}(G)$.
	In particular, we have that $\chi\in\WAP(G)|_{G_{|\omega|}}$.
\end{theorem}

\begin{proof}
In either case above, Theorem \ref{thm:fixed-in-luc} provides a character
$\chi:S_{|\omega|}\to\bt$ for which $\omega=\chi|\omega|$.  We will
proceed to show for each case that the fixed points of $L_\omega$ are contained in the set $\luc_{\chi,S_{|\omega|}}(G)$. Note that  Lemma
\ref{lem:CD-ext-to-char} shows that we may  deem from the beginning
$\chi$ to be a character on $G_{|\omega|}$. 
	
	{ (i)}  We let $f\in \Fix \lt_\omega\setminus\{0\}$ and write $S=\supp|\omega|$.  Define for $n \in \bn$ and $t \in G$
	\begin{align*}
	g_n(t)=\int_{S^{\times n}}|\chi(s_1\dots s_n)f(s_1\dots s_nt)-
	\chi(s_1\dots s_{n-1})f(s_1\dots s_{n-1}t)|^2\,
	d|\omega|^{\times n}(s_1,\dots,s_n) 
	\end{align*}
	Expanding the integrand,
        exploiting the fact that $G$ is abelian
        (so that $s_1\dots s_nt = s_ns_1\dots s_{n-1}t$)
        and that $f\in \Fix \lt_\omega$ we see that
        $g_n=\lt_{|\omega|^{\conv n}}(|f|^2)- 
	\lt_{|\omega|^{\conv (n-1)}}(|f|^2)$. Hence the telescopic series $\sum_{n=1}^\infty g_n$ converges (as increasing and bounded),
	so that $\lim_{n\to \infty}g_n=0$.
	On the other hand,  the Cauchy-Schwarz inequality tells us that for each $t \in G$
	\begin{align*}
	&\int_{ S^{\times (n-1)} }|\chi(s_1\dots s_n)f(s_1\dots s_nt)-
	\chi(s_1\dots s_{n-1})f(s_1\dots s_{n-1}t)|^2 \,d|\omega|^{\times (n-1)}(s_1,\dots,s_{n-1})  \\
	&\geq \left|\int_{S^{\times (n-1)}}[\chi(s_1\dots s_n)f(s_1\dots s_nt)-
	\chi(s_1\dots s_{n-1})f(s_1\dots s_{n-1}t)]\,d|\omega|^{\times (n-1)}(s_1,\dots,s_{n-1})\right|^2 \\
	&=|\chi(s_n)f(s_nt)-f(t)|^2.
	\end{align*} 
        Hence an application of the Fubini--Tonelli theorem shows that
        $g_n\geq g_1$.  Thus $g_1=0$, which 
	tells us for any $t$ that $\chi(s)f(st)=f(t)$ for $|\omega|$-a.e.\ $s$, and hence, by continuity,
	for every $s\in S$. In other words $\ell_s f=\conj{\chi(s)}f$ for $s\in S$.  A
	simple induction shows this holds for $s\in\bigcup_{n=1}^\infty S^n$ and hence for all  $s \in S_{|\omega|}$. 
	
	{(ii)} Let $f\in \Fix \lt_\omega\cap\WAP(G)\setminus\{0\}$. 
	Consider the product probability space
	\[
	(\Omega,\mathcal{B}_\infty,P)=(G^\infty,\mathcal{B}_\infty,|\omega|^{\times\infty})
	\]
	where $\mathcal{B}_\infty$ is the $\sigma$-algebra generated by the sequence
	of $\sigma$-algebras given for each  $n\in \bn$ by 
	\[
	\mathcal{B}_n=\sigma\langle B_1\times\dots\times B_n\times
	G\times G\times\dots :\text{ each }B_i\text{ is a Borel set in }G\rangle.
	\]
	Let $\mathcal{L}(\mathcal{B}_n)=L^1(\Omega,\mathcal{B}_n,P;\WAP(G))$
	for each $n=1,2\dots,\infty$.  We have condititional 
	expectations $\mathbb{E}_n:\mathcal{L}(\mathcal{B}_\infty)\to\mathcal{L}(\mathcal{B}_n)$ given by
	\[
	[\mathbb{E}_n(\alpha)](s_1,s_2,\dots)=\int_{\Omega}
	\alpha(s_1,\dots,s_{n-1},s_n,s_{n+1},\dots)\,dP(s_{n+1},s_{n+2},\dots).
	\]
	Define $\beta_n$ in $\mathcal{L}(\mathcal{B}_n)$ for $P$-a.e.\ $(s_1,s_2,\dots)$ by
	\[
	\beta_n(s_1,s_2,\dots)=\chi(s_1\dots s_n)\ell_{s_n\dots s_1}f
	\]
	and observe that $\|\beta_n\|_{L^1}\leq\|f\|_{\infty}$.  Also, observe that for $P$-a.e.\
	$(s_1,s_2,\dots)$ we have
	\begin{align}
	\|\chi(s_n)\ell_{s_n}f-f\|_{\infty}&=
	\|\chi(s_1\dots s_{n-1})l_{s_{n-1}\dots s_1}[\chi(s_n)\ell_{s_n}f-f]\|_{\infty}  =\|(\beta_n-\beta_{n-1})(s_1,s_2,\dots)\|_{\infty} \label{eq:betadif}
	\end{align}
	We compute, for any $n\geq 2$, $t \in G$ and $(s_1,s_2,\ldots)$ as above  that
	\begin{align*}
	\left([\mathbb{E}_{n-1}(\beta_n)](s_1,s_2, \ldots)\right)(t)
	&=\chi(s_1\dots s_{n-1})\ell_{s_{n-1}\dots s_1}\int_G f(s_nt)\chi(s_n)\,d|\omega|(s_n) \\
	&=\chi(s_1\dots s_{n-1})\ell_{s_{n-1}\dots s_1} f(t)=(\beta_{n-1}(s_1,s_2,\dots))(t).
	\end{align*}
	Hence $(\beta_n)_{n=1}^\infty$ is a martingale in $\mathcal{L}(\mathcal{B}_\infty)$.  
Moreover, the essential range of each $\beta_n$ is contained in
$\bt\{l_g f: g \in G\}$,
and thus is relatively weakly compact as $f\in\WAP(G)$.  Hence by \cite[VIII Theo.~2]{chat2}, there is $\beta$ in 
	$\mathcal{L}(\mathcal{B})$, where $(\Omega,\mathcal{B},\bar{P})$ is the completion of 
	$\mathcal{B}_\infty$, for which $\beta=\lim_{n\to\infty}\beta_n$, $\bar{P}$-a.e., 
	and each $n \in \bn$ we have $\beta_n=\mathbb{E}_n(\beta)$.  
	By \cite[Theo.~1]{chat1}, $\lim_{n\to\infty}\|\beta_n-\beta\|_{L^1}=0$. 
	Thus for $P$-a.e.\ $(s_1,s_2,\dots)$ we have, using (\ref{eq:betadif}), that
\begin{align*}
  \int_G\|\chi(s_n)\ell_{s_n}f-f\|_\infty\,d|\omega|(s_n)
  &=\int_\Omega \|\beta_n-\beta_{n-1}\|_\infty\,dP
 =\|\beta_n-\beta_{n-1}\|_{L^1}\overset{n\to\infty}{\longrightarrow}0
	\end{align*}
and hence $\int_G\|\chi(s)\ell_s f-f\|_\infty\,d|\omega|(s)=0$, so $\ell_s f=\overline{\chi(s)}f$
for $|\omega|$-a.e.\ $s$.  As in the last part of (i) above we see that
	$f\in\WAP_{\chi,S_{|\omega|}}(G)$.
	
	Notice that by taking right translations, we may assume that our non-zero fixed point
	$f$ satisfies $f(e)\not=0$.  Thus for $s$ in $G_{|\omega|}$ we must have
	$f(s)=\conj{\chi(s)}f(e)$, so that $\conj{\chi}\in\WAP(G)|_{G_{|\omega|}}$.
	Clearly the algebra $\WAP(G)_{G_{|\omega|}}$ of functions on $G_{|\omega|}$ is self-adjoint.
\end{proof}

We remark that the proof of (i) shows that if $\omega=\chi|\omega|$
for a continuous character 
on $S_{|\omega|}$, then $\ell_s h=\conj{\chi(s)}h$ for $s$ in
$S_{|\omega|}$ and every bounded  Borel-measurable function $h$.

If $G$ is a semisimple Lie group with finite center, 
then it is shown in \cite{veech} (\cite{chou} for $G=\mathrm{SL}_2(\mathbb{R})$)
that for any non-compact closed one-parameter subgroup $R\cong\mathbb{R}$ 
we have $\WAP(G)|_R=C_0(\mathbb{R})\oplus\bc 1$.  Hence there are many
measures $\omega$ in $\M(G)_1$ for which $\Fix\lt_\omega\cap\WAP(G)=\{0\}$.

If we ask about the fixed points in $C_0(G)$, also here in the classical case we obtain a more refined result than Proposition
\ref{prop:more-equiv}.

\begin{cor}\label{cor:c_naught}
Let $G$ be a locally compact group and $\omega\in\M(G)_1$.  Then, the
following are equivalent: 
\begin{rlist}
\item $\Fix \lt_\omega\cap C_0(G)\not=\{0\}$;
\item $G_{|\omega|}$ is compact and there is a continuous character
  $\chi:G_{|\omega|}\to\bt$ such that $\omega=\chi|\omega|$; 
\item  Ces\`aro sums, $S_n(\omega)$, do not converge weak* to $0$.
\end{rlist}
In this case we have that
$\chi m_{G_{|\omega|}} = \wlim_{n\to\infty}S_n(\omega)$, and
that 
\[
\Fix \lt_\omega\cap C_0(G) = \Fix \lt_{\chi m_{G_{|\omega|}}}\cap
C_0(G) =\lt_{\chi m_{G_{|\omega|}}}(C_0(G)). 
\]
In particular, if $G$ is compact and $|\omega|$ is adapted, then
$\Fix \lt_\omega = \C \overline{\chi}$. 
\end{cor}

\begin{proof}
The equivalence of (i) and (iii) is from Proposition
\ref{prop:more-equiv}.  It follows from 
Theorem \ref{thm:fixed-abelian-and-wap} that (i) is equivalent to
(ii), where the compactness 
of $G_{|\omega|}$ follows form the fact that if $f$ is a non-zero element of 
$C_0(G)\cap\WAP_{\chi,G_{|\omega|}}(G)$ (see the terminology introduced in the last theorem)
then $|f|$ is constant on cosets of $G_{|\omega|}$.  Since
$\chi m_{G_{|\omega|}}$ is an idempotent, a rudimentary
computation shows that
$C_{0,\chi,G_{|\omega|}}(G):=\{f\in C_0(G):f(st)=\conj{\chi(s)}f(t)\text{ for }s\in G_{|\omega|},t\in G\}=\lt_{\chi m_{G_{|\omega|}}}(C_0(G))
=\Fix \lt_{\chi m_{G_{|\omega|}}}\cap C_0(G)$.
	
Finally, by Lemma~\ref{MGlimit} $\wt{\omega} = \wlim_{n\to\infty} S_n(\omega)$ 
is a non-zero contractive idempotent with
$C_{0,\chi,G_{|\omega|}}(G)=\Fix\lt_{\tilde{\omega}}\cap C_0(G)$. 
Hence by \cite{greenleaf:homo}, $\tilde{\omega}=\theta m_K$ for
some compact group $K$ and a continuous character $\theta:K\to\bt$.  
But $\Fix\lt_{\theta m_K}\cap C_0(G)=\lt_{\theta m_K}(C_0(G))=C_{0,\theta,K}(G)$.
We conclude that $K=G_{|\omega|}$ and 
$\theta=\chi$, and hence $\wt\omega = \chi m_{G_{|\omega|}}$.
\end{proof}

Note that for $\omega$ in $\Prob(G)$, it is shown in \cite{kawadaito} that 
$m_{G_\omega}= \wlim_{n\to\infty}S_n(\omega)$.

\vspace*{0.2cm}

We finish with some observations generalising the results of this section to a more abstract context, and present one more application.

Let $E$ be a Banach space.  A strong operator continuous
representation $\pi:G\to\mathcal{GL}(E)$ will be called
{\it essentially weakly almost periodic} if there is a separating subspace
$F\subseteq E^*$ such that
\[
M_{\pi,F}=\linsp\{\langle f,\pi(\cdot)x\rangle:x\in E,f\in F\}\subseteq \WAP(G).
\]
i.e.\ the space generated by matrix coefficients of $\pi$ with $F$ consists of
weakly almost periodic functions.  Notice that the uniform boundedness principle
shows that $\pi$ must have bounded range.  If $\omega\in\M(G)$, then we define
the operator $\pi(\omega)$ in terms of Bochner integrals:
$\pi(\omega)x=\int_G\pi(s)x\,d\omega(s)$ for $x$ in $E$. 

\begin{cor}\label{cor:ecwap_rep}
Let $G$ be a locally compact group, $\omega\in\M(G)_1$, and
$\pi:G\to\mathcal{GL}(E)$ 
be an essentially weakly almost periodic representation.  If
$\pi(\omega)$ has a non-zero fixed point, then there is a continuous character
$\chi:G_{|\omega|}\to\bt$ such that $\omega = \chi |\omega|$ and
\[
\Fix \pi(\omega)=\{x\in E:\pi(s)x=\overline{\chi(s)}x\text{ for all }s\text{ in }G_{|\omega|}\}.
\]
In particular, we have that $\bar{\chi}\in M_{\pi,F}|_{G_{|\omega|}}$,
where $F$ is as in the definition of essential weak almost periodicity
of $\pi$.
\end{cor}

\begin{proof}
Let $\check{\pi}(s)=\pi(s^{-1})$ for $s$ in $G$, so $\check{\pi}$ is
an anti-homomorphism.   
If $x\in E$ and $f\in F$, let $\pi_{f,x}=\la f,\pi(\cdot)x\ra$ .
We also let $\check{\omega}$ denote the unique measure which satisfies
$\int_G g\,d\check{\omega}=\int_G\check{g}\,d\omega$ where
$\check{g}(s)=g(s^{-1})$ 
for $g$ in $\WAP(G)$.  Notice that
$\check{\pi}_{f,x}=\la f,\check{\pi}(\cdot)x\ra$. 
We then have for $t$ in $G$ that
\[
L_{\check{\omega}} \check{\pi}_{f,x}(t)
=\int_G\la f,\check{\pi}(s^{-1}t)x\ra\,d\omega(s)
=\int_G\la f,\check{\pi}(t)\pi(s)x\ra\,d\omega(s)
=\check{\pi}_{f,\pi(\omega)x}(t).
\]
Since $F$ is separating, it follows that, for $x\in E$, we have 
$x\in\Fix\pi(\omega)$ if and only if 
$\check{\pi}_{f,x}\in \Fix \lt_{\check{\omega}}\cap \WAP(G)$ for every
$f$ in $F$. 
Suppose that $\Fix\pi(\omega)\not=\{0\}$.
By Theorem \ref{thm:fixed-abelian-and-wap} (ii),
$\check{\omega}=\bar{\chi}|\check{\omega}|=\bar{\chi}|\omega|^\vee$
for some continuous character $\chi$ on 
$G_{|\check{\omega}|}=G_{|\omega|}$, and hence $\omega=\chi|\omega|$.
Furthermore,
\[
\check{\pi}_{f, x} \in \WAP_{\bar{\chi},G_{|\omega|}}(G)
\]
whenever $x\in \Fix\pi(\omega)$ and $f\in F$.
Hence, for $x$ in $E$, we have that $x\in\Fix\pi(\omega)$
exactly when for $f$ in $F$ and $s$ in $G_{|\omega|}$ we have
\[
\la f,\pi(s)x\ra=\check{\pi}_{f,x}(s^{-1})
=\chi(s^{-1})\check{\pi}_{f,x}(e)=\overline{\chi(s)}\pi_{f,x}(e)=\la f,\overline{\chi(s)}x\ra
\]
and, again since $F$ is separating, this happens exactly when $\overline{\chi(s)}x=\pi(s)x$.  In this case, a variant of the
last computation also shows that $\overline{\chi(s})\pi_{f,x}(e)=\pi_{f,x}(s)$, so $\bar{\chi}\in M_{\pi,F}|_{G_{|\omega|}}$.
\end{proof}

Finally  consider the left regular representation
$\lambda_p:G\to \mathcal{GL}(L^p(G))$, 
$\lambda_p(s)f(t)=f(s^{-1}t)$ for each $s\in G$ and almost every $t\in G$.  As before we let $H^p_\omega(G)=\Fix\lambda_p(\check\omega)$.  In the
commutative setting (i.e\ the case of locally compact groups), we gain
the following improvement to  Theorem \ref{prop:Lpharmonic}. 

\begin{cor}
	Let $G$ be a locally compact group, let $p\in [1,\infty)$, and let 
          $\omega\in \M(G)_1$. 
	Suppose that $H^p(G)\not=\{0\}$.  Then $G_{|\omega|}$ is
        compact, $\omega=\chi|\omega|$ 
	for some continuous character $\chi:G_{|\omega|}\to\bt$, and
	\[
	H^p_\omega(G)= \Ran \lambda_p(\overline{\chi} m_{G_{|\omega|}})
	\]
	where $m_{G_{|\omega|}}$ is normalised Haar measure on $G_{|\omega|}$.
	In particular, if $\omega$ is adapted, then $H^p_\omega(G)=\bc\bar{\chi}$.
\end{cor}

\begin{proof}
	Letting $F=L^{p'}(G)$ (dual space) for $p>1$, or $C_0(G)$ for $p=1$, we see that 
	$M_{\lambda_p,F}\subseteq C_0(G)\subseteq \WAP(G)$.
	We then appeal directly to Corollaries \ref{cor:ecwap_rep} and \ref{cor:c_naught}.
\end{proof}


\section{Mukherjea condition and the dual of the classical case}


We devote the last section to a discussion of possible generalizations in the `dual to classical' case 
of the following classical result due to Mukherjea \cite{Muk} (see also \cite{Der}),
and draw several consequences.  

\begin{tw} \label{MukhSums} 
	Let $G$ be a locally compact group 
	and let $\omega \in \Prob(G)$.  Then the following are equivalent:
\begin{rlist}
\item Ces\`aro sums $S_n(\omega)$ converge to $0$ weak$^*$ on $C_0(G)$,
\item convolution iterates $\omega^{\conv n}$ converge to $0$ weak$^*$
  on $C_0(G)$, 
\item the group  $G_\omega$ is not compact.
\end{rlist}
\end{tw}

The result is due to Derriennic (Theorem 8 of \cite{Der}) who assumed that $\omega$ is adapted.
However, the form above is equivalent since extension holds, i.e.\ $C_0(G)_{G_\omega}=C_0(G_\omega)$.
Mukherjea established this result earlier in \cite{Muk} under the second countability assumption.

In this section we consider the cocommutative case, i.e.\ the case when $\G$ is the dual of
a locally compact group $G$. In this case, $C_0^u(\G)^*$ is identified
with the Fourier--Stieltjes algebra $B(G)$ (see \cite{eymard}),
which is the linear span of $P(G)$ -- the cone of continuous positive
definite  functions on $G$. Moreover, in this case  $\linf(\G)$ is
the group von Neumann algebra $\vn(G)$, the preadjoint of which
is the Fourier algebra $A(G)$ (again, see \cite{eymard}).
As we will show below in Theorem \ref{Mukherjeedual}, the analogous
result of Theorem~\ref{MukhSums} holds in the dual case, 
if we interpret condition (iii) in the following way: the pre-image of
$1$ with respect to the normalised function $\omega \in \B(G)^+$ is not open.

We cannot hope to extend Theorem \ref{MukhSums}  to contractive functionals as
the implication (ii)$\Longrightarrow$(iii)  can fail -- there exist 
nilpotent contractive measures on compact groups -- as can
(i)$\Longrightarrow$(ii) -- consider simply $\omega =-\delta_e$. We
will still see in Theorem \ref{Mukherjeedual} that there is a simple way to
cut out such `pathological' examples. 

Finally note that in the classical context (iii)$\Longrightarrow$(ii) holds for $\omega \in \M(G)_1$, as follows from the above result and a simple convolution estimate
$|\omega^{\conv n}| \leq |\omega|^{\conv n}$, $n \in \bn$. 

For $\omega\in B(G)$, write
\[
Z_\omega = \set{s\in G}{\omega (s) = 1}.
\]

\begin{theorem} \label{thm:Zproperties}
  Suppose that $G$ is a locally compact group and
  $\omega \in \B(G)_1$.
  \begin{rlist}
  \item
    If $Z_{\omega}$ is empty, then $\Fix L_{\omega} =\{0\}$.
    This happens in particular when $\|\omega\|< 1$.
\item
  If  $Z_{\omega}$ is not empty, then $Z_\omega = s H$ where
  $H$ is a closed subgroup of $G$ and $s\in G$.
  Moreover, $\Fix L_{\omega} = \pi(s) \VN(H)$.
  \end{rlist}
  In particular, $\Fix L_{\omega}$ is always a $W^*$-sub-TRO of
  $\VN(G)$, i.e.\ $\Fix L_{\omega}$ is a weak*-closed subspace
  that is closed under the ternary product $(x, y, z)\mapsto xy^*z$.
  
\end{theorem}

\begin{proof}
 (i)
  Suppose that $x\ne 0$ is a fixed point of $L_\omega$.
  Let $s\in \supp x$ (see \cite[D\'efinition~4.5]{eymard} for the
  definition of $\supp x$)
  and pick $f\in A(G)$ such that $f(s)\ne 0$.
  Now
  \[
  0 = L_f(x - L_\omega(x)) = L_{f - f\omega}(x).
  \]
  Therefore, since $f - f\omega\in A(G)$, it follows from
  Proposition 4.4 of \cite{eymard} that $(f - f\omega)(s)= 0$.
  Since $f(s)\ne 0$, we have that $\omega(s) = 1$.
  So if $Z_\omega = \emptyset$, then $\Fix L_\omega = \{0\}$.

  The latter statement is clear because
  if $\|\omega\|< 1$, then $|\omega(s)|< 1$ for every $s\in G$.
  
  (ii)  Let $s\in Z_\omega$ and define $\tau(t) = \omega(s t)$ for
  $t\in G$. The $\tau\in \B(G)_1$ and $\tau(e) = 1$ so $\tau$ is a state.
Let $\pi\col G\to M(C^*(G))$ be the natural map. 
For every $t\in Z_\tau$ also $t\inv\in Z_\tau$ as $\tau$ is positive
definite and  
\[
\tau(\pi(t)^*\pi(t)) = \tau (t\inv t) = 1 = 
\conj{\tau(t)}\tau(t) = \tau(\pi(t)^*)\tau(\pi(t)).
\]
Hence $\pi(t)$ is in the multiplicative domain of $\tau$
(considered as a state on $M(C^*(G))$) and it follows that 
$Z_\tau$ is closed under multiplication. 
Therefore $H:= Z_\tau$ is a closed subgroup of $G$
and $Z_\omega = s Z_\tau = sH$.

Let $t\in G$. Then $\lt_\omega(\pi(t)) = \omega(t)\pi(t) = \pi(t)$ 
if and only if $t\in sH$. It follows that 
$\pi(s) \VN(H) \sub \Fix \lt_{\omega}$. 
Conversely, if $x\in \Fix \lt_\omega$, then the proof of the statement (i)
shows that $\supp x \sub Z_\omega$, and so
$\Fix \lt_{\omega} \sub \pi(s) \VN(H)$.
\end{proof}

\begin{cor}
Suppose that $\omega \in \B(G)_1$ is non-degenerate. 
Then $\dim (\Fix L_{\omega})\leq 1$. Moreover, if $\Fix L_{\omega}$ is
non-zero, it contains a unitary. 
\end{cor}

\begin{proof}
  We shall show that
  if $\omega$ is non-degenerate, then  $Z_{|\omega|} \subset \{e\}$. 
Suppose otherwise and pick $s\in Z_{|\omega|}\sm\{e\}$. Then 
\[
|\omega|^{\conv k}\bigl(((\pi(s)-\pi(e))^*(\pi(s)-\pi(e))\bigr)
= 2|\omega|^{\conv k}(e) - |\omega|^{\conv k}(s)-|\omega|^{\conv k}(s\inv) = 0
\]
for every $k\in\N$, which is in contradiction with non-degeneracy.
The preceding theorem implies that
if $\Fix L_\omega \ne \{0\}$, 
then $\Fix L_\omega = \pi(s) \complex$ for some $s\in G$
as $Z_{|\omega|} = H = \{e\}$ (in the notation of that theorem).
Now $\pi(s)$ is the required unitary element.
\end{proof}

Before we proceed with the main result of this section, we record
below how Theorem~\ref{thm:Zproperties} provides an answer to a
question asked by Chu and Lau, who
posed the following in \cite[Remark 3.3.16]{Chu-Lau}:
Let $\omega \in B(G)_1$ be such that  $Z_\omega$ is discrete 
and $A(G)/{I_\omega}$ has a bounded approximate identity
(recall that $I_\omega$ denotes the pre-annihilator of 
$\Fix L_\omega$ in $A(G)$; see Definition~\ref{def:preann}).  
Does Arens regularity of $A(G)/{I_\omega}$ imply finiteness of
$Z_\omega$? We now show that this is indeed the case.

\begin{theorem} 
Let $G$ be a locally compact group, and let $\omega \in B(G)_1$ be
such that $Z_\omega$ discrete. Assume that the Banach algebra
$A(G)/{I_\omega}$ has a BAI and is Arens regular. Then $Z_\omega$ is
finite.  
\end{theorem}

\begin{proof} 
  By assumption, we are in the situation of
  \cite[Proposition 3.3.15 (ii)]{Chu-Lau}. Hence,  
  $E := I_\omega^\perp = \Fix L_\omega$ is, as a
  	Banach space, isomorphic to a Hilbert space, so in particular
        $E$ is reflexive.  
	We can assume that $Z_\omega$ is non-empty. 
	By Theorem~\ref{thm:Zproperties} (ii),
        we have $Z_\omega = sH$ and $E = \lambda_s \VN(H)$,
        where $H$ is a closed subgroup of $G$ and $s \in G$. Thus, the
        Banach space  $\VN(H)$ is isomorphic to $E$,
	hence reflexive. But a von Neumann algebra whose underlying
        Banach space is reflexive, is finite-dimensional; cf., e.g.,
        \cite[Proposition 1.11.7]{lbr}.  So $H$ is finite, whence
        $Z_\omega = sH$ is finite as well.  
\end{proof} 

We now return to the main theme of the section.

\begin{theorem} \label{Mukherjeedual}
  If $\omega\in B(G)$ such that $\|\omega\| = 1$ and
  $Z_\omega\ne \emptyset$, then the
following are equivalent: 
\begin{rlist}
\item Ces\`aro sums $S_n(\omega)$ converge to $0$ weak$^*$ on $C^*(G)$;
\item convolution iterates $\omega^{\conv n}$ converge to $0$ weak$^*$
  on $C^*(G)$;
\item the set $Z_{\omega}$ is not open
  \textup{(}equivalently, $Z_{\omega}$ is locally null\textup{)}.
\end{rlist}
In particular, the above statements are equivalent when $\omega$ is a state.
\end{theorem}

\begin{proof}
Note first that by Theorem~\ref{thm:Zproperties} the
set $Z_{\omega}$ is a coset of a closed subgroup $H$ of $G$. Naturally
$Z_{\omega}$ is open if and only if $H$ is open, and locally null with
respect to the left Haar measure of $G$ if and only if $H$ is locally null. Now
the fact that $H$ is not locally null is equivalent to $H$ being open is
standard (and follows from a Steinhaus type theorem, see for example
Corollary 20.17 in \cite{HR}). We now proceed with the proof of the
main equivalences.  
  
(ii)$\Longrightarrow$(i): trivial.
 
(i)$\Longrightarrow$(iii): 
  assume that (iii) is not true. Since $Z_\omega$ is open,
  there is a compactly supported continuous function $f$ 
  on $G$ such that $\supp f \sub Z_\omega$ and 
  $\int f(s)\, ds = 1$ (note that $Z_\omega \ne \emptyset$ by
  assumption).   
  If $\pi\col L^1(G)\to C^*(G)$ is the canonical map, 
  then 
  \[
  \omega^{\conv k} (\pi(f)) = 
\int f(s) \omega(s)^{k}\, ds = \int f(s)\, ds = 1
\]
for every $k\in \bn$. 
Hence $S_n(\omega)(\pi(f)) = 1$ for every $n\in \bn$ 
and (i) is not true. 

(iii)$\Longrightarrow$(ii): let then $\omega$ be as in the assumptions
of the theorem.  
The proof of Theorem~\ref{thm:Zproperties} shows that there
exists $s \in Z_\omega$ and a state $\tau \in B(G)$ such that
$\tau(t)=\omega(st)$, $t \in G$; moreover $Z_\omega = sZ_\tau$ (so
that (iii) is equivalent to $Z_\tau$ being null). If $f \in L^1(G)$
then  
\[
\omega^{\conv n}(\pi(f)) = \int_G \omega^n(t) f(t) dt,\qquad
\tau^{\conv n}(\pi(f)) = \int_G \omega^n(st) f(t)  dt
= \int_G\omega^n(t) f(s^{-1}t) dt.
\]
This implies that (ii) is equivalent to the convolution iterates $\tau^{\conv n}$ converging to $0$ weak$^*$
on $C^*(G)$. 

The discussion above implies that  we can assume
without loss of generality
that $\omega$ is a state. There is then a strongly continuous unitary
representation $\pi:G \to B(\Hil)$ on some Hilbert space $\Hil$ and a
unit vector $\xi \in \Hil$ such that $\omega(g) = \langle \xi, \pi(g)
\xi \rangle, g \in G$. Write $U:=\omega^{-1}(\mathbb{T})$. As for any
$z \in \mathbb{T}$ and $g \in G$ we have 
\[
\omega(s) = z \iff z = \langle \xi, \pi(g) \xi \rangle
\iff \pi(g) \xi = z \xi,
\]
it is easy to see that $U$ is a closed subgroup of $G$ and $\omega|_U
:U \to \mathbb{T}$ is a homomorphism with kernel equal to
$Z_\omega$. Thus $\omega|_U =\chi \circ q$, where $q:U\to U/Z_\omega$
is the quotient map and $\chi\in \widehat{U/Z_\omega}$. Further
$U/Z_\omega \cong \omega(U) \subset \mathbb{T}$. 
 
Let $f\in L^1(G)$. If $U$ is not open, then $\mu_G(U)=0$, as argued in
the beginning of the proof, and we have
\[
\omega^{\conv n}(\pi(f)) =\int_{G\setminus U} \omega^n(t) f(t) dt \stackrel{n\to
  \infty}{\longrightarrow} 0
\]
by the dominated convergence theorem.
We may then assume that $U$ is open;  and then $U/Z_\omega$ cannot be a
finite subgroup of $\mathbb{T}$ (as $Z_\omega$ is assumed not
open). Hence $U/Z_\omega \cong  \mathbb{T}$, and $\chi\neq 1$. By the
standard integration formula we have 
\begin{align*}
  \int_{U} \omega^n(t) f(t) dt
  &= \int_{U/Z_\omega} \left(\int_{Z_\omega} \omega(st)^n f(st)
  d\mu_{Z_\omega}(t)  \right)  d\mu_{U/Z_\omega} (sZ_\omega)  \\
  &=\int_{U/Z_\omega} \chi(sZ_\omega)^n\left(\int_{Z_\omega}  f(st)
  d\mu_{Z_\omega}(t) \right)  d\mu_{U/Z_\omega} (sZ_\omega),
\end{align*}
and the latter tends to $0$ by the Riemann-Lebesgue lemma -- as the
function $sZ_\omega \mapsto \int_{Z_\omega}  f(st) d\mu_{Z_\omega}(t)$
is integrable. This ends the proof.  
\end{proof}

Note that the above theorem is related to Theorem 3 of
\cite{Mehrdad2}, which is on one hand only treating the
positive-definite functions $\omega$, but on the other obtains a
stronger notion of convergence in (ii), i.e.\ the point--norm
convergence of the convolution iterates. 

Furthermore, we note that the equivalence of (i) and (iii) in
Theorem~\ref{Mukherjeedual} above also
follows from the recent, independently obtained result in
\cite[Proposition 2.5]{Mus}.
However, the equivalence with (ii) in Theorem~\ref{Mukherjeedual} does
not follow from \cite{Mus}, as the example of $u(t) = e^{it}$ on $G=\mathbb{R}$
shows, which is covered by our Theorem~\ref{Mukherjeedual},
but does not satisfy the assumption of
\cite[Proposition 4.4 (a)]{Mus}; cf.\ the remark after the proof of the latter.

Finally we stress that the condition that $Z_\omega$ is not open is
connected to the fact that the `support' of $\omega$ is
non-compact. We see this clearly in the case where $G$ is abelian. 

\begin{prop}
  Let $G$ be a locally compact abelian group
  and let $\mu \in \M(G), \|\mu\|=1$. Then
$Z_{\hat{\mu}}\neq \emptyset$ if and only if 
$\frac{d \mu}{d |\mu|} =\chi$ $\mu$-a.e. for some character  $\chi \in
\hat{G}$. Moreover, in 
that case $G_\mu$ is compact if and only if $Z_{\hat{\mu}}$ is open. 
\end{prop}

\begin{proof}
If $\chi\in Z_{\hat\mu}$, then 
\[
1 = \hat\mu(\chi) = \int \conj{\chi(s)}\, d\mu(s) 
= \int \conj{\chi(s)}\frac{d\mu}{d|\mu|}(s)\, d|\mu|(s), 
\]
which implies that $\frac{d \mu}{d |\mu|} =\chi$ $\mu$-a.e.
(because $\|\mu\| = 1$). The converse is trivial. 

Moreover, if $\chi \in Z_{\hat\mu}$, then 
\begin{align*}
\eta\in Z_{\hat\mu} &\iff 
1 = \int \conj{\eta(s)}\chi(s)\, d|\mu|(s) \\
&\iff \conj{\eta}\chi = 1 \text{ on } \supp \mu \\
&\iff \conj{\eta}\chi = 1 \text{ on } G_\mu \\
&\iff \eta \in \chi G_\mu^\perp.
\end{align*}
Then the second statement follows from the fact that 
$G_\mu$ is compact if and only if $G_\mu^\perp$ is open.
\end{proof}

Note that the preceding  two results imply  Mukherjea's
result for locally compact abelian groups.

\subsection*{Acknowledgements} The work on this article was initiated during the Research in Pairs
visit to the Mathematisches Forschungsinstitut Oberwolfach in August
2012. We are very grateful to MFO for providing ideal research
conditions.
AS was partially supported by the National Science Center (NCN) grant no.~2014/14/E/ST1/00525. 
PS was partially supported by the Emil Aaltonen Foundation. We are grateful to the referee for a careful reading of our paper and several useful comments.



\begin{thebibliography}{Run09}


\bibitem[BMS]{BBMS}
G. A. Bagheri-Bardi, A. R. Medghalchi and N. Spronk, Operator-valued convolution algebras, 
\emph{Houston J. Math.} \textbf{36} (2010), no. 4, 1023--1036.

\bibitem[BrR]{br-ru} 
M.\ Brannan and Z.-J.\ Ruan, 
$L_p$-representations of discrete quantum groups, 
\textit{J.\ Reine Angew.\ Math.} \textbf{732} (2017), 165--210. 

\bibitem[Cha$_1$]{chat1}
S. D. Chatterji, A note on the convergence of Banach-space valued martingales,
\emph{Math. Ann.} \textbf{153} (1964) 142--149.

\bibitem[Cha$_2$]{chat2}
S. D. Chatterji,  Les martingales et leurs applications analytiques, 
\emph{\'{E}cole d'\'{E}t\'{e} de Probabilit\'{e}s: Processus Stochastiques (Saint Flour, 1971)}, pp. 27--164. 
Lecture Notes in Math., Vol. 307, Springer, Berlin, 1973.

\bibitem[ChD]{ChD}
G. Choquet and J. Deny, Sur l'\'{e}quation de convolution $\mu=\mu\ast\sigma$, 
\emph{C. R. Acad. Sci. Paris} \textbf{250} (1960), 799--801.

\bibitem[Cho]{chou}
C. Chou, Weakly almost periodic functions and Fourier-Stieltjes algebras of locally compact groups. 
\emph{Trans. Amer. Math. Soc.} \textbf{274} (1982), no. 1, 141--157.







\bibitem[Chu]{chu} 
C.-H.\ Chu, 
Harmonic function spaces on groups, 
\textit{J.\ London Math.\ Soc.} (2) \textbf{70} (2004), no.\ 1, 182--198. 

\bibitem[ChL]{Chu-Lau} C.H.\,Chu and A.T.-M.\,Lau, ``Harmonic functions on groups and Fourier algebras,''
Lecture Notes in Mathematics, 1782. Springer-Verlag, Berlin, 2002.

\bibitem[Coh]{Cohen} P.\,Cohen, On homomorphisms of group algebras, \emph{Amer.\,J.\,Math.} \textbf{82} (1960), 213--226.

\bibitem[Daw]{daws} M.\,Daws,
Completely positive multipliers of quantum groups,
\emph{Internat.\,J.\,Math.} \textbf{23} (2012), 1250132, 23 pp.


\bibitem[Der]{Der} Y.\,Derriennic,   Lois ``z\'ero ou deux'' pour les processus de Markov. Applications aux marches al\'eatoires, \emph{Ann.\,Inst.\,H.\,Poincar\'e} \textbf{12} (1976), no.\,2, 111--129. 

\bibitem[Eym]{eymard}
P. Eymard, \emph{L'alg\`ebre de {F}ourier d'un groupe localement compact},
   {Bull. Soc. Math. France}, \textbf{92} (1964), 181--236.



\bibitem[Gre]{greenleaf:homo}
F.~P. Greenleaf, Norm decreasing homomorphisms of group algebras, \emph{Pacific J.\,Math.} \textbf{15} (1965), 1187--1219.


\bibitem[Grn]{Gren}
U.\,Grenander, ``Probabilities on Algebraic Structures,'' John Wiley \& Sons, New-York-London 1963.

\bibitem[HeR]{HR} E.\,Hewitt and K.\,Ross, ``Abstract harmonic analysis. Vol. I. Structure of topological groups, integration theory, group representations.'', Second edition. Grundlehren der Mathematischen Wissenschaften [Fundamental Principles of Mathematical Sciences], \textbf{115}. Springer-Verlag, Berlin-New York, 1979.


\bibitem[HNR$_1$]{HNRSIN} Z.\,Hu, M.\,Neufang and Z.-J.\,Ruan, On topological centre problems and SIN quantum
groups, \emph{J.\,Funct.\,Anal.} \textbf{257} (2009), 610--640.

\bibitem[HNR$_2$]{HNR2} Z.\,Hu, M.\,Neufang, and Z.-J.\,Ruan,
  Completely bounded multipliers over
locally compact quantum groups, \emph{Proc. London Math. Soc.}
\textbf{103} (2011), 1--39.

\bibitem[Jaw]{Jaworski} W.\,Jaworski, Ergodic and mixing probability measures on [SIN] groups, \emph{J.\,Theoret.\,Probab.}
\textbf{17} (2004), 741--759.

\bibitem[JaN]{jaw-neu} 
W.\ Jaworski and M.\ Neufang, 
The Choquet--Deny equation in a Banach space, 
\textit{Canad.\ J.\ Math.} \textbf{59} (2007), no.\ 4, 795--827. 


  \bibitem[JNR]{jnr} M.\,Junge, M.\,Neufang and Z.-J.\,Ruan, A representation theorem for locally compact quantum groups,
\emph{Internat.\,J.\,Math.} \textbf{20} (2009), 377--400.

\bibitem[Kai]{Vadim} V.A.\,Kaimanovich,  Boundaries of invariant Markov operators: the identification problem, \emph{in} Ergodic theory of $\mathbb{Z}^d$-actions (Warwick, 1993–1994), pp. 127--176, London Math.\,Soc.\,Lecture Note Ser., \textbf{228}, Cambridge Univ.\,Press, Cambridge, 1996.


\bibitem[Ka$_1$]{Mehrdad1} M.\,Kalantar, A limit theorem for discrete quantum groups, \emph{J.\,Funct.\,Anal.} \textbf{265} (2013), no.\,3, 469--473.

\bibitem[Ka$_2$]{kal-lp} 
M.\ Kalantar, 
On harmonic non-commutative $L_p$-operators on locally compact quantum groups, 
\textit{Proc.\ Amer.\ Math.\ Soc.} \textbf{141} (2013), no.\,11, 3969--3976. 

\bibitem[Ka$_3$]{Mehrdad2} M.\,Kalantar, On iterated powers of positive definite functions,\emph{Bull.\,Aust.\,Math.\,Soc.} \textbf{92} (2015), no.\,3, 440--443.

\bibitem[KNR$_1$]{knr} M.\,Kalantar, M.\,Neufang and Z.-J.\,Ruan, Poisson boundaries over locally compact quantum groups, \emph{Int.\,J.\,Math.} \textbf{24} (2013), 1350023.

\bibitem[KNR$_2$]{knr2} M.\,Kalantar, M.\,Neufang and Z.-J.\,Ruan, Realization of quantum group Poisson boundaries as crossed products,  \emph{Bull.\,Lond.\,Math.\,Soc.} \textbf{46} (2014), no.\,6, 1267--1275.

\bibitem[Kas]{kasprzak} P. Kasprzak, Shifts of group-like projections
  and contractive idempotent functionals for locally compact
  quantum groups, \emph{Internat.\,J.\,Math.} \textbf{29} (2018), no.\,13, 1850092, 19 pp.
  
  
  
  \bibitem[KI]{kawadaito}
  Y. Kawada and K. It\^o, 
  On the probability distribution on a compact group. I.
  \emph{Proc. Phys.-Math. Soc. Japan (3)}, \textbf{22}, (1940) 977--998.
  
\bibitem[Kus]{ku} J. Kustermans,
Locally compact quantum groups in the universal setting,
\emph{Internat.\,J.\,Math.} \textbf{12} (2001), no.\,3, 289--338.

\bibitem[KV]{KV} J.\,Kustermans and S.\,Vaes, Locally compact quantum groups,
\emph{Ann.\,Sci.\,{\'E}cole Norm.\,Sup.\,(4)} \textbf{33}
(2000), no.\,9,  837--934.

\bibitem[Li]{lbr} 
B.R.\ Li, 
``Introduction to operator algebras,'' 
\textit{World Scientific Publishing Co., Inc.}, River Edge, NJ, 1992. 

\bibitem[Mas]{masuda} 
T.\ Masuda, 
$L_p$-spaces for von Neumann algebra with reference to a faithful normal semifinite weight, 
\textit{Publ.\ Res.\ Inst.\ Math.\ Sci.} \textbf{19} (1983), no.\,2, 673--727. 




\bibitem[Muk]{Muk} A. Mukherjea, Limit theorems for probability measures on non-compact groups and semigroups, \emph{Z.\,Wahrscheinlichkeitstheorie Verw.\,Gebiete} {\bf 33} (1975/1976), no.\,4, 273--284.


\bibitem[Mus]{Mus}
  H. Mustafayev, Mean ergodic theorems for multipliers on Banach
algebras, \emph{J. Fourier Anal. Appl.} {\bf 25} (2019), no.\,2, 393.--426.


\bibitem[NaT]{NakTak} Y.\,Nakagami and M.\,Takesaki, ``Duality for Crossed Products of von Neumann Algebras,'' Lecture Notes in Mathematics, 731. Springer, Berlin-Heidelberg-New York, 1979.


\bibitem[NSSS]{NSSS} M.\,Neufang, P.~Salmi, A.~Skalski and N.~Spronk, Contractive idempotents on locally compact quantum groups, \emph{Indiana Univ.\,Math.J.} \textbf{62 }(2013), no.\,6, 1983--2002.


\bibitem[Rau]{raugi}
A. Raugi, A general Choquet-Deny Theorem for nilpotent groups,  \emph{Ann. Inst. H. Poincar\'{e} 
	Probab. Statist.} \textbf{40} (2004), no.\,6, 677--683.

\bibitem[Run]{runde}
  V. Runde,  Uniform continuity over locally compact quantum groups,
  \emph{J. London Math. Soc.} \textbf{80} (2009), no.\,1, 55--71.
  
 
 \bibitem[Sal]{Pekkasurvey} P.\,Salmi, Idempotent states on locally compact
 groups and quantum groups,
 \emph{Algebraic Methods in Functional Analysis,
 	The Victor Shulman Anniversary Volume},
 Operator Theory: Advances and Applications, Vol. 233,
 I.G. Todorov and L. Turowska (Eds.), pp. 155-170,
 Birkh\"auser/Springer, Basel, 2014 

\bibitem[SaS$_1$]{salmi-skalski:idem}
P.~Salmi and A.~Skalski, Idempotent states on locally compact quantum  groups, \emph{Quart.\,J.\,Math.} {\bf 63} (2012), no.\,4, 1009--1032.


\bibitem[SaS$_2$]{PekkaAdamKyoto}
P.~Salmi and A.~Skalski, Actions of locally compact (quantum) groups on ternary rings of operators, their crossed products, and generalized Poisson boundaries, \emph{Kyoto J.\,Math.} \textbf{57} (2017), no.\,3, 667--691. 


\bibitem[Tak]{takesaki:vol1}
M.~Takesaki, \emph{Theory of operator algebras. {I}}, Encyclopaedia of
  Mathematical Sciences, vol. 124, Springer-Verlag, Berlin, 2002.
  
  \bibitem[Tem]{temp}
  A. Tempelman, ``Ergodic theorems for group actions. Informational and thermodynamical aspects,''
  Mathematics and its Applications, 78. Kluwer Academic Publishers Group, Dordrecht, 1992.
  
  \bibitem[Vee]{veech}
  W. A. Veech, Weakly almost periodic functions on semisimple Lie groups,
  \emph{Monatsh. Math.} \textbf{88} (1979), no. 1, 55--68.

\bibitem[Yau]{yau} 
S.T.\ Yau, 
Some function-theoretic properties of complete Riemannian manifold and their applications to geometry, 
\textit{Indiana Univ.\ Math.\ J.} \textbf{25} (1976), no.\ 7, 659--670. 

\bibitem[Wi]{willard}
  S. Willard, \emph{General Topology} Addison-Wesley,
  Reading, Mass.-London-Don Mills, Ont. 1970.

\bibitem[Wil]{Willis}
G.\,Willis, Probability measures on groups and some related ideals in group algebras, \emph{J.\,Funct.\,Anal.} \textbf{92} (1990), no.\,1, 202--263.

\bibitem[Won]{Wong} J.C.S.\,Wong, Invariant means on locally compact semigroups, \emph{Proc.\,Amer.\,Math.\,Soc.} \textbf{31} (1972), no.\,1, 39--45.



\end{thebibliography}
\end{document}